\documentclass[psamsfonts]{amsart}

\usepackage{amsmath,amssymb,amsthm,enumerate,tikz-cd,comment,url}
\usepackage[colorlinks=false]{hyperref}

\theoremstyle{plain}
\newtheorem*{main-theorem}{Main Theorem}
\newtheorem{theorem}{Theorem}[subsection]
\newtheorem{lemma}{Lemma}[subsection]
\newtheorem{corollary}{Corollary}[subsection]
\newtheorem{proposition}{Proposition}[subsection]
\newtheorem*{orlov}{Orlov's Theorem}

\theoremstyle{definition}
\newtheorem{definition}{Definition}[section]
\newtheorem{example}{Example}[section]

\theoremstyle{remark}
\newtheorem{remark}{Remark}[section]

\DeclareMathOperator{\Bl}{\mathrm{Bl}}
\DeclareMathOperator{\Ext}{\mathrm{Ext}} 
\DeclareMathOperator{\Hom}{\mathrm{Hom}}
\DeclareMathOperator{\Spec}{\mathrm{Spec}}

\begin{document}

\title[Equivariant Derived Categories of Certain Hypersurfaces]{Equivariant derived
categories associated to a sum of two potentials}

\author[B. Lim]{Bronson Lim}
\address{BL: Department of Mathematics \\ California State University San
  Bernardino \\ San Bernardino, CA 92407, USA}
\email{bronson.lim@csusb.edu}

\maketitle

\begin{abstract}
	Suppose \(f,g\) are homogeneous polynomials of degree \(d\) defining smooth
	hypersurfaces \(X_f = V(f)\subset\mathbb{P}^{m-1}\) and \(X_g =
	V(g)\subset\mathbb{P}^{n-1}\). Then the sum of \(f\) and \(g\)
	defines a smooth hypersurface \(X = V(f\oplus g)\subset\mathbb{P}^{m+n-1}\)
	with an action of \(\mu_d\) scaling the \(g\) variables. Motivated by the
	work of Orlov, we construct a semi-orthogonal decomposition of the derived
	category of coherent sheaves on \([X/\mu_d]\) provided
	\(d\geq\mathrm{max}\{m,n\}\).
\end{abstract}

\section{Introduction}
\label{sec:intro}

\subsection{Semi-orthogonal decompositions in algebraic geometry}
\label{ssec:sod-ag}

To a space \(X\), i.e. a smooth and projective variety or more generally a
smooth and proper Deligne-Mumford stack, we can associate the bounded derived
category of coherent sheaves on the space, denoted \(\mathcal{D}(X)\). The
category \(\mathcal{D}(X)\) lives in the intersection between homological
algebra and algebraic geometry and has proved to be a useful tool 
when applied to algebro-geometric problems.

Of particular interest is when \(\mathcal{D}(X)\) admits a semi-orthogonal
decomposition (see Section \ref{ssec:sod} for the definition). Roughly, a
semi-orthogonal decomposition is the analogue of a group extension for
triangulated categories. If \(\mathcal{D}(X)\) admits a semi-orthogonal
decomposition, one can hope to further understand \(\mathcal{D}(X)\), or
sometimes \(X\), using the components of the decomposition. 

See the surveys \cite{bondal-orlov-sod} and \cite{bridgeland-derived} for
examples of semi-orthogonal decompositions and their uses.

\subsection{Orlov's Theorem}
\label{ssec:orlov-theorem}

Let \(k\) be an algebraically closed field of characteristic zero and \(V\) a
vector space over \(k\) of dimension \(n\). Assume \(f\in k[V]_d\) defines a smooth
hypersurface, say \(X_f=V(f)\subset\mathbb{P}(V)\). We call \(f\) a
\textit{potential}. Let \(\mathrm{HMF}^{gr}(f)\) denote the homotopy category of
graded matrix factorizations of the potential \(f\). Recall, objects of
\(\mathrm{HMF}^{gr}(f)\) are \(\mathbb{Z}/2\)-graded, curved complexes of
\(\mathbb{G}_m\)-equivariant vector bundles on \(V\) with curvature \(f\). There
is a natural differential on the space of morphisms between two matrix
factorizations. The category \(\mathrm{HMF}^{gr}(f)\) is the corresponding
homotopy category.

A relationship between \(\mathrm{HMF}^{gr}(f)\) and \(\mathcal{D}(X_f)\) was
discovered by Orlov in \cite{orlov-sing-09}. Orlov constructs two
\(\mathbb{Z}\)-indexed families of exact functors \(\Psi_i:\mathcal{D}(X_f)\to
\mathrm{HMF}^{gr}(f)\) and \(\Phi_i:\mathrm{HMF}^{gr}(f)\to \mathcal{D}(X_f)\).
If \(X_f\) is Fano or Calabi-Yau, then \(\Phi_i\) is a full embedding. If
\(X_f\) is general type or Calabi-Yau, then \(\Psi_i\) is a full embedding.
Moreover, the semi-orthogonal complement is determined:

\begin{orlov} \cite[Theorem 3.11]{orlov-sing-09}
	Let \(f\) be as above. For each \(i\in\mathbb{Z}\), we have the following
	semi-orthogonal decompositions:
	\begin{align*}
		\text{Fano} &: \mathcal{D}(X_f) = \langle
		\mathcal{O}_{X_f}(-i-n+d+1),\ldots,\mathcal{O}_{X_f}(-i),\Phi_i\mathrm{HMF}^{gr}(f)\rangle;
		\\
		\text{General Type} &: \mathrm{HMF}^{gr}(f) = \langle
		k^{stab}(-i),\ldots,k^{stab}(-i+n-d+1),\Psi_i\mathcal{D}(X_f)\rangle; \\
		\text{Calabi-Yau} &: \Phi_i,\Psi_i\text{ induce mutual inverse equivalences }
		\mathcal{D}(X_f)\cong \mathrm{HMF}^{gr}(f).
	\end{align*}
	Here \(k^{stab}\) is a certain matrix factorization associated to the residue
	field of \(k[V]\) at the origin.
\end{orlov}

\subsection{Adding two potentials.}
Let \(f,g\) be homogeneous polynomials of degree \(d\) defining smooth
hypersurfaces \(X_f\subset\mathbb{P}^{m-1}\) and \(X_g\subset\mathbb{P}^{n-1}\).
Let \(X = V(f\oplus g)\subset\mathbb{P}^{m+n-1}\). Then \(X\) is smooth since
\(X_f\) and \(X_g\) are smooth.

Suppose \(d\geq \mathrm{max}\{m,n\}\). Then there is a \(\mathbb{Z}\)-indexed
family of embeddings \(\Psi_i:\mathcal{D}(X_f)\to \mathrm{HMF}^{gr}(f)\) and
similarly \(\Psi_j:\mathcal{D}(X_g)\to \mathrm{HMF}^{gr}(g)\). By tensoring, we
can consider the family of embeddings:
\[
	\Psi_{i,j}\colon\mathcal{D}(X_f\times X_g)\cong \mathcal{D}(X_f)\otimes
	\mathcal{D}(X_g) \to \mathrm{HMF}^{gr}(f)\otimes \mathrm{HMF}^{gr}(g).
\]
where \(\Psi_{i,j} = \Psi_i\otimes\Psi_j\) and the
tensor product is understood to be taken in suitable
dg-enhancements\footnote{It is known that Orlov's functors lift to the dg level
\cite{cal-tu-13}}. We further have an identification, see \cite[Corollary 5.18]{bfk1-14}
\[
	\mathrm{HMF}^{gr}(f)\otimes \mathrm{HMF}^{gr}(g)\cong
	\mathrm{HMF}^{gr,\mu_d}(f\oplus g).
\]
where \(\mu_d\) acts on the \(g\) variables.

If in addition \(d\leq n+m\), it was noticed in \cite[Example 3.10]{bfk2-14}
that we can then embed \(\mathrm{HMF}^{gr,\mu_d}(f\oplus g)\) into
\(\mathcal{D}[X/\mu_d]\) using Orlov's Theorem
a second time.  Fix one such embedding to get a doubly indexed family of fully-faithful
functors \(\Xi_{i,j}\colon\mathcal{D}(X_f\times X_g)\to
\mathcal{D}[X/\mu_d]\). The complement consists of \(mn\) exceptional
objects, \(d-m\) copies of \(\mathcal{D}(X_g)\), and \(d-n\) copies of
\(\mathcal{D}(X_f)\). Specifically, we have
\[
	\mathcal{D}[X/\mu_d] = \langle \mathcal{A}, \mathcal{K}, \mathcal{D}_f, \mathcal{D}_g,
	\mathcal{D}_{fg}\rangle
\]
where \(\mathcal{A}\) consists of \( (m+n-d)d\) line bundles,
\(\mathcal{K} \cong \langle k^{stab}(-i),\ldots,k^{stab}(-i+m-d+1)\rangle\otimes \langle
k^{stab}(-j),\ldots,k^{stab}(-j+n-d+1)\rangle\), \(\mathcal{D}_f =
\Phi_i\mathcal{D}(X_f)\otimes\langle k^{stab}(-j),\ldots,k^{stab}(-j+n-d+1)\rangle\),
\(\mathcal{D}_g = \langle k^{stab}(-i),\ldots, k^{stab}(-i+m-d+1)\rangle\otimes
\Phi_j\mathcal{D}(X_g)\), and \(\mathcal{D}_{fg} =
\Xi_{i,j}\mathcal{D}(X_f\times X_g)\).

These functors are not easy to compute with however and, with the exception of
\(\mathcal{A}\), explicitly understanding the left and right semi-orthogonal
complements to the image of \(\Xi_{i,j}\) as \(\mu_d\)-equivariant complexes of
sheaves on \(X\) is not easy. 

\subsection{Main result}

In this paper, we give a more geometric definition of the functors \(\Xi_{i,j}\)
and show that they, miraculously, remain embeddings even if \(d>n+m\). Moreover,
we explicitly determine the other components in the associated semi-orthogonal
decomposition.

\begin{main-theorem}
  Assume \(n\geq m\), then there is a semi-orthogonal decomposition
	\[
		\mathcal{D}[X/\mu_d] = \langle \mathcal{D}_{g1}, \mathcal{D}_{fg},
		\mathcal{D}_{g2}, \mathcal{D}_f, \mathcal{A}\rangle,
	\]
  where \(\mathcal{D}_{g1}\) and \(\mathcal{D}_{g2}\) collectively consist of
  \(d-m\) twists of \(\mathcal{D}(X_g)\) (Section \ref{sssec:d-xg}),
  \(\mathcal{D}_f\) consists of \(d-n\) twists of \(\mathcal{D}(X_f)\) (Section
  \ref{sssec:d-xf}), \(\mathcal{D}_{fg}\) is the image of \(\Xi_{-m,-n}\)
  (Section \ref{ssec:embedding-xf-xg}), and \(\mathcal{A}\) consists of an
  exceptional collection of \({m\choose 2} + m(n-m) +{m+1\choose 2}\) line
  bundles (Section \ref{ssec:a-category}).
\end{main-theorem}

To align with the picture given by Orlov's theorem we can mutate the
decomposition; however, it gets complicated quickly. As stated, each of the
components has a simple description given by explicit Fourier-Mukai functors. 

\subsection{Outline of paper}

In Section \ref{sec:prelims}, we recall facts about (equivariant)
triangulated categories. Section \ref{sec:geometry-pn} is devoted to
understanding the derived category of the quotient stack
\([\mathbb{P}^{m+n-1}/\mu_d]\). In Section \ref{sec:ts} we define all of the
terms in the above decomposition. In Section \ref{sec:sod} we prove 
semi-orthogonality. In Section \ref{sec:koszul} we discuss
sheaves constructed by taking Koszul complexes. In Section \ref{sec:full} we
prove fullness. In Section \ref{sec:orlov}, we show that in the case
\(X_f\) and \(X_g\) are Calabi-Yau, our functors agree with Orlov's up to a
twist by a line bundle. We end the paper with Section
\ref{sec:special-cases}, which is devoted to the special case when \(m=1\).

\subsection{Acknowledgements}

This work was completed at the University of Oregon as the authors thesis. The
author is very grateful to Alexander Polishchuk for introducing this problem to
him and for guiding him through the graduate program at the University of
Oregon. The author is also grateful to the referee for their substantial
comments and suggestions for improvement.

\section{Preliminaries on (Equivariant) Triangulated Categories}
\label{sec:prelims}

Throughout \(k\) is an algebraically closed field of characteristic zero. For an
overview of triangulated categories in algebraic geometry see
\cite{huybrechts-fourier}.

\subsection{Triangulated categories}

Recall, a triangulated category \(\mathcal{T}\) is a  \(k\)-linear category
together with an autoequivalence \([1]:\mathcal{T\to T}\) and a class of exact
triangles
\[
	t\to u\to v\to t[1]
\]
satisfying certain axioms, see \cite{gelfand-manin}.

\begin{example}
	If \(X\) is a scheme, or more generally an algebraic stack, we can associate a
	triangulated category called the bounded derived category of coherent sheaves
	on \(X\), denoted \(\mathcal{D}(X)\). This category is the Drinfeld-Verdier
	localization of the category of chain complexes of coherent sheaves with
	bounded cohomology with respect to the class of quasi-isomorphisms.
	\label{ex:dcat-coherent-sheaves}
\end{example}

\subsection{Semi-orthogonal decompositions}
\label{ssec:sod}

Let \(\mathcal{T}\) be a triangulated category. A \textbf{semi-orthogonal
decomposition} of \(\mathcal{T}\), written
\[
	\mathcal{T} = \langle \mathcal{A}_1,\ldots,\mathcal{A}_n\rangle
\]
is a sequence of full triangulated subcategories
\(\mathcal{A}_1,\ldots,\mathcal{A}_n\) of \(\mathcal{T}\) such that:
\begin{itemize}
	\item \(\mathrm{Hom}_{\mathcal{T}}(a_i,a_j) = 0\) for
		\(a_i\in\mathcal{A}_i\), \(a_j\in \mathcal{A}_j\), and \(i>j\);
	\item For any \(t\in \mathcal{T}\), there is a sequence of morphisms
		\[
			0 = a_n\to a_{n-1}\to\cdots \to a_1\to a_0 = t
		\]
		where \(\mathrm{Cone}(a_i\to a_{i-1})\in\mathcal{A}_i\).
\end{itemize}

\begin{example}
	The most common examples of semi-orthogonal decompositions occur when
	\(\mathcal{T} = \mathcal{D}(X)\) for some smooth projective scheme over \(k\).
	In this case, there is often a vector bundle \(\mathcal{E}\) such that
	\(\mathrm{Ext}^\ast_X(\mathcal{E},\mathcal{E})\cong k[0]\). Such an object in
	\(\mathcal{D}(X)\) is called \textbf{exceptional}. The subcategory generated
	by \(\mathcal{E}\) is also abusively denoted by \(\mathcal{E}\) and there are
	two semi-orthogonal decompositions (that it exists follows from Example
	\ref{ex:saturated-exceptional-object}):
	\[
		\mathcal{D}(X) = \langle \mathcal{E}^\perp,\mathcal{E}\rangle = \langle
		\mathcal{E},{}^\perp\mathcal{E}\rangle
	\]
	where \(\mathcal{E}^\perp = \{\mathcal{F}^\cdot\in\mathcal{D}(X)\mid
	\mathrm{Ext}^\ast_X(\mathcal{E},\mathcal{F}^\cdot) = 0\}\) and
	\({}^\perp\mathcal{E}\) is defined similarly.
	\label{ex:exceptional-objects}
\end{example}

\subsection{Spanning classes}

Let \(\mathcal{T}\) be a triangulated category. A subclass of objects
\(\Omega\subset\mathcal{T}\) is called a \textbf{spanning class} if for every
\(t\in \mathcal{T}\) the following two conditions hold:
\begin{itemize}
	\item \(\mathrm{Hom}_{\mathcal{T}}(t,\omega[i]) = 0\) for all \(\omega\in
		\Omega\) and all \(i\in\mathbb{Z}\) implies \(t=0\);
	\item \(\mathrm{Hom}_{\mathcal{T}}(\omega[i],t) = 0\) for all \(\omega\in
		\Omega\) and all \(i\in\mathbb{Z}\) implies \(t=0\).
\end{itemize}

\begin{example}
	If \(X\) is a smooth projective variety over \(k\), then a spanning class is
	furnished by the structure sheaves of closed points:
	\[
		\Omega = \{\mathcal{O}_x\mid x\in X\text{ is a closed point}\}.
	\]
	More generally, if \(\mathcal{X}\) is a smooth and proper Deligne-Mumford
  stack over \(k\), then for each \(k\)-point \(x\colon\mathrm{Spec}(k)\to
  \mathcal{X}\), there is a closed embedding of the residual gerbe \(\iota_x\colon
  [\mathrm{Spec}(k)/\mathrm{Aut}(x)]\hookrightarrow \mathcal{X}\). For each
  irreducible representation \(\xi\) of \(\mathrm{Stab}(x)\), we have an object
  \(\mathcal{O}_{x,\xi}:=\iota_{x\ast}(\mathcal{O}_k\otimes V)\in\mathcal{D(X)}\). The collection of all such objects, as
  \(x\) and \(\xi\) vary, form a spanning class, see \cite[Prop. 2.1]{LP-20}:
	\[
    \Omega = \{\mathcal{O}_{x,\xi}\mid
      x\colon\mathrm{Spec}(k)\to\mathcal{X}\text{ and
    }\xi\in\mathrm{Irr}(\mathrm{Aut}(x))\}
	\]
	\label{ex:spanning-class-stack}
\end{example}

\subsection{Admissible triangulated subcategories}

Let \(\mathcal{A}\subset\mathcal{T}\) be a full triangulated subcategory of a
triangulated category. We say \(\mathcal{A}\) is
\textbf{admissible} if the embedding functor \(\iota:\mathcal{A\to T}\) has a left and right
adjoint.\footnote{We will not need the more general notions of left and right
admissibilility in this paper.}

If \(\mathcal{A}\) is admissible, then it follows formally that \(\mathcal{T}\)
admits two semi-orthogonal decompositions
\[
	\mathcal{T} = \langle \mathcal{A}^\perp,\mathcal{A}\rangle = \langle
	\mathcal{A}, {}^\perp\mathcal{A}\rangle.
\]
where
\begin{align*}
	\mathcal{A}^\perp := \{t\in \mathcal{T}\mid
	\mathrm{Hom}_{\mathcal{T}}(t,a[i]) = 0\text{ for all
	}a\in\mathcal{A},i\in\mathbb{Z}\}; \\
	{}^\perp \mathcal{A}:= \{t\in\mathcal{T}\mid
	\mathrm{Hom}_{\mathcal{T}}(a[i],t) = 0\text{ for all
	}a\in\mathcal{A},i\in\mathbb{Z}\}.
\end{align*}

We have the following useful lemma regarding admissible subcategories.

\begin{proposition}
	Suppose \(\Omega\) is a spanning class for \(\mathcal{T}\) and \(\mathcal{A}\)
	is a full, admissible, triangulated subcategory containing \(\Omega\), then \(\mathcal{A}
	= \mathcal{T}\).
	\label{prop:subcat-spanning-class}
\end{proposition}

\begin{proof}
	Since \(\mathcal{A}\) is admissible, there is a semi-orthogonal decomposition
	of \(\mathcal{T}\) of the form:
	\[
		\mathcal{T} = \langle \mathcal{A}^\perp,\mathcal{A}\rangle
	\]
	The condition that \(\mathcal{A}\) contains a spanning class implies that
	\(\mathcal{A}^\perp\) must be trivial.
\end{proof}

\begin{proposition}
  Suppose \(\mathcal{A,B}\subset\mathcal{T}\) are admissible triangulated
  subcategories and \(\Omega_\mathcal{A}\) and \(\Omega_\mathcal{B}\) be
  spanning classes of \(\mathcal{A,B}\), respectively. If
  \(\mathrm{Hom}_\mathcal{T}(b,a[i]) = 0\) for all
  \(a\in\Omega_\mathcal{A}\), \(b\in\Omega_\mathcal{B}\), and
  \(i\in\mathbb{Z}\), then \(\langle \mathcal{A,B}\rangle\)
  is a semiorthogonal decomposition.
  \label{prop:sod-spanning}
\end{proposition}

\begin{proof}
  Since \(\mathrm{Hom}_\mathcal{T}(b,a[i]) = 0\) for all
  \(a\in\Omega_\mathcal{A}\), \(b\in\Omega_\mathcal{B}\), and
  \(i\in\mathbb{Z}\), it follows that \(\mathcal{B}\subset
  {}^\perp\mathcal{A}\).
\end{proof}

\subsection{Saturated triangulated subcategories}

A triangulated category \(\mathcal{T}\) is called \textbf{saturated} if every
cohomological functor (contravariant or covariant) \(H:\mathcal{T}\to
\mathrm{Vect}_k\) of finite type is representable. We have the following
important proposition regarding saturated subcategories, see \cite[Proposition
2.6]{bondal-kapranov-serre}.

\begin{proposition}
	Let \(\mathcal{A}\) be a saturated triangulated category and
	\(\iota:\mathcal{A}\to \mathcal{T}\) is a full embedding. Then \(\mathcal{A}\)
	is an admissible sucategory of \(\mathcal{T}\).
	\label{prop:sat-subcat-admissible}
\end{proposition}

\begin{example}
	The derived category of coherent sheaves on a smooth projective variety,
	\(X\), is saturated, \cite[Theorem 2.14]{bondal-kapranov-serre}. If
	\(\mathcal{E}\) is an exceptional object of \(\mathcal{D}(X)\), then there is
	a full embedding \(\iota_{\mathcal{E}}:\mathcal{D}(\mathrm{Spec}(k))\to
	\mathcal{D}(X)\) given by \(\iota_{\mathcal{E}}(V) = \mathcal{E}\otimes V\).
	This justifies Example \ref{ex:exceptional-objects}.
	\label{ex:saturated-exceptional-object}
\end{example}

We will use the following proposition in conjuction with Proposition
\ref{prop:subcat-spanning-class} in Section \ref{sec:full}.

\begin{proposition}\cite[Theorem 2.10]{bondal-kapranov-serre}
	If \(\mathcal{A},\mathcal{B}\subset \mathcal{T}\) are full, saturated, triangulated
	subcategories such that \(\langle \mathcal{A},\mathcal{B}\rangle\) is
	semi-orthogonal, then \(\langle \mathcal{A},\mathcal{B}\rangle\) is saturated.
	\label{prop:sod-saturatedness}
\end{proposition}

We will also need the following theorem in Section \ref{sec:full}.

\begin{theorem}
	Suppose \(F:\mathcal{D}(X)\to \mathcal{T}\) is a full embedding where \(X\) is
	smooth and projective over \(k\). Further suppose there exists a saturated subcategory
	\(\mathcal{A}\) containing \(F(\Omega)\), where \(\Omega\) is a spanning class
	for \(\mathcal{D}(X)\). Then \(F(\mathcal{D}(X))\subset\mathcal{A}\).
	\label{thm:swapping-functors}
\end{theorem}

\begin{proof}
	It is sufficient to prove that the right adjoint \(G:\mathcal{T}\to
	\mathcal{D}(X)\) is zero on \(\mathcal{A}^\perp\). Suppose \(b\in
	\mathcal{A}^\perp\). For every \(\omega\in \Omega\) and \(i\in\mathbb{Z}\) we
	have
	\[
		\Ext^i_X(\omega,G(b)) \cong
		\mathrm{Hom}_{\mathcal{A}}(F(\omega),b[i]) = 0.
	\]
	Since \(\Omega\) is a spanning class, we conclude \(G(b) = 0\) for all \(b\in
	\mathcal{A}^\perp\).
\end{proof}

\subsection{Equivariant triangulated categories}

Suppose \(G\) is a finite group. An action of \(G\) on a triangulated category
\(\mathcal{T}\) is the following data, \cite[\S 3.1]{kuznetsov-perry}:
\begin{itemize}
	\item For every \(g\in G\), an exact autoequivalence \(g^\ast:\mathcal{T\to
		T}\);
	\item For every \(g,h\in G\), an isomorphism of functors
		\(\varepsilon_{g,h}:(gh)^\ast\xrightarrow{\sim} h^\ast\circ g^\ast\)
		satisfying the usual associativity conditions.
\end{itemize}

A \(G\)-equivariant object of \(\mathcal{T}\) is a pair \( (t,\theta)\), where
\(t\in \mathcal{T}\) and \(\theta\) is a collection of isomorphisms
\(\theta_g:t\xrightarrow{\sim} g^\ast t\) for all \(g\in G\) satifying the usual
associativity diagram. 

For any action of \(G\) on a triangulated category \(\mathcal{T}\), we can form
the category of equivariant objects of \(\mathcal{T}\) denoted
\(\mathcal{T}^G\).

\begin{remark}
	It is not true that \(\mathcal{T}^G\) is always triangulated, i.e. that the
	triangulated structure on \(\mathcal{T}\) descends to \(\mathcal{T}^G\), see
	\cite[Example 8.4]{elagin-equivariant}. 
\end{remark}

\begin{example}
	Suppose a finite group \(G\) acts on a scheme \(X\), then there is an exact
	equivalence \(\mathcal{D}[X/G]\cong \mathcal{D}(X)^G\), see \cite[Section
	3.8]{vistoli}. So in this case there is a natural triangulated structure on
	\(\mathcal{D}(X)^G\). A good reference for equivariant derived categories of
	coherent sheaves is \cite[Section 4]{bkr-mukai-01}.
\end{example}

\begin{example}
  Let \( (\mathcal{F},\theta)\) be an equivariant object in \(\mathcal{D}(X)\).
  Further, let \(\chi:G\to\mathbb{G}_m\) be a multiplicative character of \(G\).
  Define a new equivariant object \(
  (\mathcal{F}\otimes\chi,\theta\otimes\chi)\) where \(\mathcal{F}\otimes\chi =
  \mathcal{F}\) as an object of \(\mathcal{D}(X)\) but the maps
  \(\theta_g\otimes\chi:\mathcal{F}\to g^\ast\mathcal{F}\) are twisted by
  \(\chi\) via the formula \(\theta_g\otimes\chi := \chi(g)\cdot\theta_g\). Thus
  if \(\mathcal{F}\) admits one equivariant structure it can admit several distinct
  equivariant structures.
	\label{ex:twisting-reps}
\end{example}

\subsection{Bondal-Orlov fully-faithfulness criterion}

We shall need the well known fully-faithfulness criterion of Bondal and Orlov.

\begin{theorem}[Bondal, Orlov]
	Let \(X\) be a smooth projective variety over \(k\) and \(\mathcal{T}\) be a
	triangulated category. Suppose \(F:\mathcal{D}(X)\to \mathcal{T}\) is an exact
	functor with a right adjoint \(G\). Then \(F\) is fully-faithful if
	and only if for any two closed points \(x,y\in X\) we have
	\[
		\mathrm{Hom}_{\mathcal{T}}(F(\mathcal{O}_x),F(\mathcal{O}_y)[i]) =
		\begin{cases}
			k & \text{ if }x=y\text{ and }i = 0 \\
			0 & \text{ if }x\neq y\text{ and }i\notin[0,\dim(X)].
		\end{cases}
	\]
	\label{thm:bondal-orlov-ff}
\end{theorem}

\begin{proof}
	The proof in \cite[Proposition 7.1]{huybrechts-fourier} only requires that the
	functor \(F\) has a right adjoint.
\end{proof}

\begin{remark}
	We will use Theorem \ref{thm:bondal-orlov-ff} when \(\mathcal{T}\) is the
	derived category of a smooth and proper Deligne-Mumford stack over \(k\). In this case,
	the existence of a (left and) right adjoint is guaranteed as the relative
	dualizing sheaf exists.
\end{remark}

\subsection{C\u{a}ld\u{a}raru-Katz-Sharpe Spectral Sequence}

The last ingredient is a spectral sequence for computing extension groups
between structure sheaves of subvarieties. Suppose \(X\) is a smooth projective
variety and \(S,T\subset X\) are subvarieties such that \(W = S\cap T\) is again
smooth. 

\begin{theorem}[\cite{cks}]
  For all \(q\in \mathbb{Z}\), we have an isomorphism
  \[
    \mathcal{E}xt^q_X(\mathcal{O}_S,\mathcal{O}_T)\cong \Lambda^m\mathcal{N}_{W/T}\otimes
    \Lambda^{q-m}\tilde{N}
  \]
  where \(m = \mathrm{codim}(W\hookrightarrow T)\) and \(\tilde{N} =
  T_X|_W/(T_S|_W+T_T|_W)\) is the excess normal bundle. In particular, there is
  a convergent spectral sequence
  \[
    {}^2E^{pq} = H^p(W;\Lambda^m\mathcal{N}_{W/T}\otimes\Lambda^{q-m}\tilde{N})
    \Rightarrow \mathrm{Ext}^{p+q}_X(\mathcal{O}_S,\mathcal{O}_T)
  \]
  \label{thm:cks-ss}
\end{theorem}

Suppose, in addition, \(X\) has an action of a finite group \(G\) and \(S\) and \(T\)
are invariant subvarieties. Then, by functoriality, the isomorphism in Theorem
\ref{thm:cks-ss} is \(G\)-equivariant provided all the occuring sheaves are
equipped with the canonical equivariant structure.

\section{Derived Category of \([\mathbb{P}^{m+n-1}/\mu_d]\)}
\label{sec:geometry-pn}

For the rest of this paper, we let \(\chi:\mu_d\to\mathbb{G}_m\) denote the
standard primitive character \(\chi(\lambda) = \lambda\). 

There is an action of \(\mu_d\) on \(\mathbb{P}^{m+n-1}\), where
\(\mathbb{P}^{m+n-1}\) has coordinates \( [x_1:\ldots:x_m:y_1:\ldots:y_n]\) and
\(\mu_d\) acts by scaling the variables \(y_1,\ldots,y_n\). In terms of the
homogeneous coordinate algebra \(k[x_1,\ldots,x_m,y_1,\ldots,y_n]\), the
variables \(y_i\) have weight \(\chi^{-1}\) and the variables \(x_i\) have
trivial weight.

In this section, we will study the corresponding quotient stack
\([\mathbb{P}^{m+n-1}/\mu_d]\) and its derived category
\(\mathcal{D}[\mathbb{P}^{m+n-1}/\mu_d]\).

\subsection{Equivariant objects}

Let \(H_y = V(x_1,\ldots,x_m)\) and \(H_x = V(y_1,\ldots,y_n)\). The fixed
locus of the \(\mu_d\) action is \((\mathbb{P}^{m+n-1})^{\mu_d} = H_x\sqcup H_y\).
Therefore the sheaves \(\mathcal{O}_{H_x}\) and \(\mathcal{O}_{H_y}\) have a natural
equivariant structure given by the identity morphism. As in Example
\ref{ex:twisting-reps}, we can form the equivariant sheaves
\(\mathcal{O}_{H_x}\otimes\chi^i\) and \(\mathcal{O}_{H_y}\otimes\chi^i\) for
\(i = 0,\ldots,d-1\). 

We equip \(\mathcal{O}(-1)\) with the \(\mu_d\)-linearization
\(\theta_\lambda\colon \mathcal{O}(-1)\to \lambda^\ast\mathcal{O}(-1)\) given by
fiberwise multiplication by \(\lambda\) and consider \(\mathcal{O}(i)\)
with the induced \(\mu_d\)-linearizations. We can also twist these sheaves by
characters to get the equivariant line bundles
\(\mathcal{O}_{\mathbb{P}^{m+n-1}}(i)\otimes\chi^j\) for \(i\in\mathbb{Z}\) and
\(j=0,\ldots,d-1\).

\subsection{Serre duality}
\label{ssec:serre}

The canonical bundle on \(\mathbb{P}^{m+n-1}\) is \(\mathcal{O}(-m-n)\).  It is
locally trivial as a \(\mu_d\)-equivariant bundle; however, the identification
\(\omega_{\mathbb{P}^{m+n-1}}\cong \mathcal{O}(-m-n)\) may involve twisting by
a character. To determine the twist, we recall the Euler exact sequence on
\(\mathbb{P}^{m+n-1}\)
\[
	0\to \Omega^1\to \mathcal{O}(-1)^{\oplus m+n}\xrightarrow{\alpha} \mathcal{O}\to 0
\]
where \(\alpha = (x_1,\ldots,x_m,y_1,\ldots,y_n)\). Since the sections \(y_i\)
have weight \(-1\), the above Euler exact sequence admits the following
\(\mu_d\)-linearization:
\[
	0\to \Omega^1\to (\oplus_{i=1}^m\mathcal{O}(-1))\oplus
	(\oplus_{j=1}^n\mathcal{O}(-1)\otimes\chi^{-1})\xrightarrow{\alpha} \mathcal{O}\to 0
\]
Now taking determinants yields \(\omega_{[\mathbb{P}^{m+n-1}/\mu_d]}\cong
\mathcal{O}_{\mathbb{P}^{m+n-1}}(-m-n)\otimes\chi^{-n}\) as
\(\mu_d\)-equivariant sheaves. Serre duality therefore takes the following form:

\begin{proposition}[Serre Duality]
	For any \(\mathcal{F},\mathcal{G}\in \mathcal{D}[\mathbb{P}^{m+n-1}/\mu_d]\)
	there is a natural isomorphism
	\[
		\Ext^\ast_{[\mathbb{P}^{m+n-1}/\mu_d]}(\mathcal{F},\mathcal{G})\cong
		\Ext^{m+n-1-\ast}_{[\mathbb{P}^{m+n-1}/\mu_d]}(\mathcal{G},\mathcal{F}(-m-n)\otimes\chi^{-n}).
	\]
	\label{prop:serre-duality-projective-mud}
\end{proposition}

\subsection{Grothendieck Splitting Theorem}

The classical Grothendieck splitting theorem decomposes any vector bundle on
\(\mathbb{P}^1\) as a sum of line bundles. We have the following equivariant
version of this result which is used in Section \ref{ssec:embedding-xf-xg}.

\begin{theorem}[Equivariant Grothendieck Splitting]
	Let \(\mathcal{E}\) be a rank \(r\) vector bundle on \([\mathbb{P}^1/\mu_d]\), then there exists
	\(n_1,\ldots,n_r\) and \(j_1,\ldots,j_r\) such that
	\[
		\mathcal{E}\cong \oplus_{i=1}^r\mathcal{O}(n_1)\otimes\chi^{j_i}.
	\]
	\label{thm:equivariant-grothendieck-splitting}
\end{theorem}

\begin{proof}
	The proof is almost identical to the classical proof. The twists
	\(\chi^{j_i}\) show up when looking for an equivariant global section of
  \(\mathcal{E}(n_i)\) with \(n_i>>0\). The fact that \(\mu_d\) is Abelian guarantees that
  the irreducible representations are one-dimensional.
\end{proof}

\section{The Hypersurface \([X/\mu_d]\)}
\label{sec:ts}

Let \(X_f\subset \mathbb{P}^{m-1}\) and \(X_g\subset\mathbb{P}^{n-1}\) be smooth
degree \(d\) hypersurfaces. Let \(X = V(f\oplus g)\subset\mathbb{P}^{m+n-1}\) be
the hypersurface associated to the sum of potentials. We impose the conditions
\(d\geq n\geq m\geq 2\), i.e. the hypersurfaces \(X_f\) and \(X_g\) are Calabi-Yau or
general type and are non-empty.

The action of \(\mu_d\) on \(\mathbb{P}^{m+n-1}\) descends to \(X\) and we
consider the quotient stack \([X/\mu_d]\).  The fixed loci are given by the
intersections with \( (\mathbb{P}^{m+n-1})^{\mu_d} = H_x\sqcup H_y\):
\[
	X^{\mu_d} = X\cap (H_x\sqcup H_y) \cong X_f\sqcup X_g.
\]

\subsection{Equivariant geometry of \(X\)}

Line bundles associated to hyperplane sections \(\mathcal{O}_X(iH)\) have \(d\)
distinct equivariant structures. These equivariant line bundles are of the form
\(\mathcal{O}_X(iH)\otimes \chi^j\).

\begin{proposition}[Serre Duality]
	The triangulated category \(\mathcal{D}[X/\mu_d]\) has the Serre functor
	\((-)\otimes\mathcal{O}_X(d-m-n)\otimes \chi^{-n}[m+n-2]\).
	\label{prop:serre-duality}
\end{proposition}

\begin{proof}
	Since \([X/\mu_d]\) is a smooth substack of \([\mathbb{P}^{m+n-1}/\mu_d]\), we
	can use the adjunction formula 
	\[
		\omega_{[X/\mu_d]}\cong
		\omega_{[\mathbb{P}^{m+n-1}/\mu_d]}\otimes\mathcal{O}_{[X/\mu_d]}(d)\cong
		\mathcal{O}_X(d-m-n)\otimes\chi^{-n}.
	\]
\end{proof}

For Fano hypersurfaces it is easy to see that line bundles are exceptional. With
this extra \(\mu_d\) action, all line bundles on \([X/\mu_d]\)
are exceptional:

\begin{proposition}
	Line bundles are exceptional objects of \(\mathcal{D}[X/\mu_d]\).
	\label{prop:exceptional-line-bundles}
\end{proposition}

\begin{proof}
	It is sufficient to prove \(H^\ast(\mathcal{O}_X)^{\mu_d}\cong k\). We have an
	equivariant exact sequence on \(\mathbb{P}^{m+n-1}\):
	\[
		0\to \mathcal{O}_{\mathbb{P}^{m+n-1}}(-d)\xrightarrow{f\oplus g}
		\mathcal{O}_{\mathbb{P}^{m+n-1}}\to \mathcal{O}_X\to 0.
	\]
	The only possible nonzero cohomology groups are
	\(H^0(\mathcal{O}_X)^{\mu_d}\) and
  \(H^{m+n-2}(\mathcal{O}_X)^{\mu_d}\). Further, we have an isomorphism
	\[
		H^{m+n-2}(\mathcal{O}_X)^{\mu_d}\cong
		H^{m+n-1}(\mathcal{O}_{\mathbb{P}^{m+n-1}}(-d))^{\mu_d}.
	\]
	If \(d<m+n\), then the latter group is zero and we are finished. Suppose
	\(d\geq m+n\). Then 
	\[
		H^{m+n-1}(\mathcal{O}_{\mathbb{P}^{m+n-1}}(-d))^{\mu_d}\cong
		H^0(\mathcal{O}_{\mathbb{P}^{m+n-1}}(d-m-n)\chi^{-n})^{\mu_d}.
	\]
	The latter has a basis of monomials of the form \(x^Iy^J\) where \(I =
	(i_1,\ldots,i_m), J=(j_1,\ldots,j_n)\) and \(x^I = x_1^{i_1}\cdots
	x_m^{i_m}, y^J = y_1^{j_1}\ldots y_n^{j_n}\) such that \(|I|+|J| = d-m-n\)
	and \(|J|+n\) is a positive multiple of \(d\). It follows that \(|J| = d-n\)
	is the only possible option. Hence, \(d-m-n=|I|+|J|=|I|+d-n\) from which we
	conclude \(|I|=-m\), impossible.
\end{proof}

\begin{proposition}
  For \(0<i<m\), \(H^\ast(\mathcal{O}_X\otimes\chi^{-i}) = 0\).
	\label{prop:full-orthogonality}
\end{proposition}

\begin{proof}
	Analagous to the computation in Proposition
	\ref{prop:exceptional-line-bundles}.
\end{proof}

\subsection{Subcategory of exceptional line bundles.}
\label{ssec:a-category}

Define subcategories \(\mathcal{A}_1,\mathcal{A}_2,\mathcal{A}_3\) of
\(\mathcal{D}[X/\mu_d]\) as follows.
\begin{align*}
	\mathcal{A}_1 &= \langle
	\mathcal{O}_X(-(n-1)-(m-1))\otimes\chi^{-(n-1)}, \\
	&\mathcal{O}_X(-(n-1)-(m-1)+1)\otimes\chi^{-(n-2),-(n-1)}, \\
	&\ldots, \mathcal{O}_X(-(n-1)-1)\otimes \chi^{-(n-m)-1,\ldots,-(n-1)}
	\rangle;\\
	\mathcal{A}_2 &= \langle \mathcal{O}_X(-(n-1))\otimes
	\chi^{-(n-m),\ldots,-(n-1)},\\
	&\mathcal{O}_X(-(n-1)+1)\otimes\chi^{-(n-m)+1,\ldots,-(n-1)+1},\\
	&\ldots,
	\mathcal{O}_X(-(m-1)-1)\otimes\chi^{-1,\ldots,-(m-1)} \rangle;\\
	\mathcal{A}_3 &= \langle
	\mathcal{O}_X(-(m-1))\otimes\chi^{0,\ldots,-(m-1)},\\
  &\mathcal{O}_X(-(m-2))\otimes\chi^{0,\ldots,-(m-2)},\ldots,\mathcal{O}_X\rangle.
\end{align*}
It is understood that if \(m=n\), then \(\mathcal{A}_2\) is zero. Further, the
notation \(\mathcal{O}_X(i)\otimes\chi^{j_1,\ldots,j_k}\) means the subcategory 
generated by the exceptional objects
\(\mathcal{O}_X(i)\otimes\chi^{j_1},\ldots,\mathcal{O}_X(i)\otimes\chi^{j_k}\).
By Proposition \ref{prop:full-orthogonality}, there is a semi-orthogonal
decomposition:
\[
	\mathcal{O}_X(i)\otimes\chi^{j_1,\ldots,j_k} = \langle
	\mathcal{O}_X(i)\otimes\chi^{j_k},\ldots,\mathcal{O}_X(i)\otimes\chi^{j_1}\rangle,
\]
where \(j_1>j_2>\cdots>j_k\) and \(j_1-j_k<m\)

\begin{proposition}
	The decomposition of the subcategories \(\mathcal{A}_i\) for \(i=1,2,3\) is
	semi-orthogonal. Moreover, \(\mathcal{A} = \langle
	\mathcal{A}_1,\mathcal{A}_2,\mathcal{A}_3\rangle\) is semi-orthogonal.
	\label{prop:sod-A}
\end{proposition}

\begin{proof}
	We only show the semi-orthogonality of the decomposition for
	\(\mathcal{A}_3\). The semi-orthogonality of the decomposition for
	\(\mathcal{A}_1\) and \(\mathcal{A}_2\) is similar.

	The exceptional objects generating \(\mathcal{A}_3\) are of the form
	\(\mathcal{O}(-i_1)\otimes \chi^{-j_1}\) for \(0\leq i\leq m-1\) and \(0\leq i_1\leq
	i_2\). Pick two such objects with \(i_1\leq i_2\) and \(j_1\leq j_2\). By
	Serre Duality and the closed substack exact sequence:
	\[
		H^{m+n-2}(\mathcal{O}_X(i_1-i_2)\otimes\chi^{j_1-j_2})\cong
		H^0(\mathcal{O}_{\mathbb{P}^{m+n-1}}(d+i_2-i_1-n-m)\otimes\chi^{j_2-j_1-n}).
	\]

	We have the following inequalities:
	\begin{align*}
		& m-1\geq i_2-i_1\geq 0, \\
		& m-1\geq j_2-j_1\geq 0.
	\end{align*}

  An equivariant global section of
  \(\mathcal{O}_X(d+i_2-i_1-n-m)\otimes\chi^{j_2-j_1-n}\) is of the form
  \(x^Iy^J\) where \(|I|+|J| = d+i_2-i_1-n-m\) such that \(|J| = j_2-j_1-n+d\).
  But
	\[
		d+i_2-i_1-n-m=|I|+|J| = |I|+d+j_2-j_1-n.
	\]
	Hence, \(|I| = i_2-i_1-(j_2-j_1)-m \leq i_2-i_1-m\leq m-1-m = -1\), which is
	impossible. Thus
	\(H^\ast(\mathcal{O}_X(i_1-i_2)\otimes\chi^{j_1-j_2}) = 0\) and the
	semi-orthogonal decomposition for \(\mathcal{A}_3\) is verified.

	We now check \(\langle \mathcal{A}_2,\mathcal{A}_3\rangle\) the others
  are similar. Recall that, for \(\mathcal{A}_2\) to be non-zero, we need \(n>m\). 
	The relevant group is
	\begin{align*}
		&H^{m+n-2}(\mathcal{O}_X(i_1-(m-1)-i_2)\otimes\chi^{j_1-j_2})^{\mu_d}\\
		&\cong H^0(\mathcal{O}_{\mathbb{P}^{m+n-1}}(d+i_2-i_1-n-1)\otimes\chi^{j_2-j_1-n})^{\mu_d}
	\end{align*}
	for \(i_1=0,\ldots,m-1\), \(j_1=0,\ldots,i_1\), \(i_2 = 1,\ldots,n-m\), \(j_2
	= i_2,\ldots,(m-1)+i_2\). 
	
	If \(d+i_2-i_1-n-1<0\), there is nothing to prove. Assume \(d+i_2-i_1-n-1\geq
	0\). Let \(x^Iy^J\) be an equivariant global section. Since \(j_2-j_1-n <0\)
	we require \(|J| = d+j_2-j_1-n\). Then
	\[
		d+i_2-i_1-n-1 = |I|+|J| = |I|+j_2-j_1-n
	\]
	forces \(|I| = i_2-j_2 + j_1-i_1 -1\). However, \(i_2-j_2\leq 0\) and
	\(j_1-i_1\leq 0\) so \(|I|\leq -1\) which is impossible. This finishes the
	proof.
\end{proof}

\subsection{Geometric subcategories}
\label{ssec:geom-subcat}

\subsubsection{The subcategory \(\mathcal{D}_f\):}
\label{sssec:d-xf}

Let \(\iota_f\colon X_f\to X\) be given by \(\iota_f([x_1:\ldots:x_m]) =
[x_1:\ldots:x_m:0:\ldots:0]\). Clearly \(\iota_f\) is \(\mu_d\)-equivariant as
it coincides with a component of the fixed locus. Let \(\iota_{f\ast}\) denote
the corresponding equivariant pushforward functor \(\iota_{f\ast}\colon\mathcal{D}(X_f)\to
\mathcal{D}[X/\mu_d]\). This means first include \(\mathcal{D}(X_f)\) into the
trivial component of
\(\mathcal{D}[X_f/\mu_d] =
\bigoplus_{i=0}^{d-1}\mathcal{D}(X_f)\otimes\chi^i\), then use
the equivariant pushforward.

\begin{proposition}
	If \(d>n\), then \(\iota_{f\ast}\) is fully-faithful.
	\label{prop:xf-ff}
\end{proposition}

\begin{proof}
	We use Theorem \ref{thm:bondal-orlov-ff}. Let \(p\in X_f\) be a closed point
	and identify \(p\) with \(\iota_f(p)\). It is sufficient to show vanishing of
	\(\Ext^\ast_{[X/\mu_d]}(\mathcal{O}_p,\mathcal{O}_p)\cong (\Lambda^\ast
	T_pX)^{\mu_d}\) for \(\ast>m-2\). From the normal bundle exact sequence
	\[
		0\to TX\to T\mathbb{P}^{m+n-1}|_X\to \mathcal{O}_X(d)\to 0
	\]
	and the identification \(T_p\mathbb{P}^{m+n-1}\cong \mathbf{1}^{\oplus
	m-1}\oplus \chi^{\oplus n}\) coming from the \(\mu_d\)-linearized Euler exact sequence, we see
	\[
		T_pX\cong \mathbf{1}^{\oplus m-2}\oplus \chi^{\oplus n}
	\]
	If \(d>n\), then \( (\Lambda^\ast T_pX)^{\mu_d} = 0\) for \(\ast>m-2\).
\end{proof}

Define the following subcategories of \(\mathcal{D}(X)^{\mu_d}\):
\[
	\mathcal{D}_f^i = \iota_{f\ast}(\mathcal{D}(X_f))\otimes\chi^i.
\]

\begin{proposition}
	For \(0<i_1-i_2<d-n\) we have the semi-orthogonality
	\[
		\langle \mathcal{D}_f^{i_1}, \mathcal{D}_f^{i_2} \rangle.
	\]
	\label{prop:xf-sod}
\end{proposition}

\begin{proof}
  The isomorphism of \(\mu_d\)-representations
	\[
		\Ext^\ast_X(\mathcal{O}_p\otimes
		\chi^{i_2},\mathcal{O}_p\otimes\chi^{i_1}) \cong
		\Lambda^\ast(\mathbf{1}^{\oplus m-2}\oplus \chi^{\oplus n})\otimes
    \chi^{i_1-i_2}
	\]
  shows
  \[
    \Ext^\ast_X(\mathcal{O}_p\otimes
    \chi^{i_2},\mathcal{O}_p\otimes\chi^{i_1})^{\mu_d} \cong 0.
  \]
  Since we have semi-orthogonality on a spanning class, by Proposition
  \ref{prop:sod-spanning} there is a semi-orthogonal decomposition
  \[
    \langle \mathcal{D}_f^{i_1}, \mathcal{D}_f^{i_2} \rangle.
  \]
\end{proof}

Let \(\mathcal{D}_f\) be the strictly full subcategory of
\(\mathcal{D}[X/\mu_d]\) generated by
\(\mathcal{D}_f^1,\ldots,\mathcal{D}_f^{d-n}\). 

\begin{corollary}
	For \(d>n\), we have a semi-orthogonal decomposition 
	\[
		\mathcal{D}_f = \langle
		\mathcal{D}_f^{d-n},\mathcal{D}_f^{d-n-1},\ldots,\mathcal{D}_f^1\rangle.
	\]
	\label{cor:df-sod}
\end{corollary}

\subsubsection{The subcategory \(\mathcal{D}_g\)}
\label{sssec:d-xg}

Similarly to \(\mathcal{D}_f\), we have a closed embedding
\(\iota_g:X_g\to X\) given by \(\iota_g([y_1:\ldots:y_n]) =
[0:\ldots:0:y_1:\ldots:y_n]\), which is the inclusion of the other
component of the fixed locus and so is \(\mu_d\)-equivariant. Let
\(\iota_{g\ast}\colon\mathcal{D}(X_g)\to \mathcal{D}[X/\mu_d]\) be the associated
equivariant pushforward. The following results are analagous to Propositions
\ref{prop:xf-ff}, \ref{prop:xf-sod} and Corollary \ref{cor:df-sod}.

\begin{remark}
	The restriction of the equivariant structure on a hyperplane divisor of \(X\)
	to \(X_g\) is not the trivial structure. In particular, we have isomorphisms:
	\[
		\mathcal{O}_X(i)|_{X_g}\cong \mathcal{O}_{X_g}(i)\otimes\chi^{-i}.
	\]
  \label{rem:equivariant-structure}
\end{remark}

\begin{proposition}
	If \(d>m\), then \(\iota_{g\ast}\) is fully-faithful.
	\label{prop:xg-ff}
\end{proposition}

Define the subcategories
\[
	\mathcal{D}_g^i = \iota_{g\ast}(\mathcal{D}(X_g))\otimes\chi^i
\]
of \(\mathcal{D}([X/\mu_d])\).

\begin{proposition}
	For \(m-d<i_1-i_2<0\) we have
	\[
		\langle \mathcal{D}_g^{i_1},\mathcal{D}_g^{i_2}\rangle.
	\]
	\label{prop:xg-sod}
\end{proposition}

\begin{corollary}
	For \(d>m\) we have a semi-orthogonal decomposition
	\[
		\mathcal{D}_g = \langle \mathcal{D}_g^{m-d}, \mathcal{D}_g^{m-d+1}, \ldots,
		\mathcal{D}_g^{-1}\rangle.
	\]
	\label{cor:dg-sod}
\end{corollary}

In the case \(n>m\), it will be necessary to split \(\mathcal{D}_g\) into two
subcategories. Define:
\[
	\mathcal{D}_{g1} = \langle \mathcal{D}_g^{m-d},
	\mathcal{D}_g^{m-d+1},\ldots,\mathcal{D}_g^{m-n-1}\rangle.
\]
and
\[
	\mathcal{D}_{g2} = \langle \mathcal{D}_g^{m-n},\mathcal{D}_g^{m-n+1},\ldots,
	\mathcal{D}_g^{-1}\rangle.
\]
We have \(\mathcal{D}_g = \langle \mathcal{D}_{g1},\mathcal{D}_{g2}\rangle\),
where it is understood that if \(m=n\), then \(\mathcal{D}_g =
\mathcal{D}_{g1}\).

\subsection{Embedding \(\mathcal{D}(X_f\times X_g)\)}
\label{ssec:embedding-xf-xg}

Let \(Y = \mathbb{P}(\mathcal{O}_{X_f}(-1)\boxplus \mathcal{O}_{X_g}(-1))\) and
\(\pi\colon Y\to X_f\times X_g\) be the projection. Let \(J = V(f,g)\subset X\)
be the join of \(X_f\) and \(X_g\) viewed as subvarieties of \(X\). Then \(Y\) is a resolution of
\(J\). Let \(\sigma\colon Y\to X\) be the induced mapping. Define \(\iota\colon
Y\to X_f\times X_g\times X\) via \(\iota = (\pi,\sigma)\). Then the following
diagram is commutative:
\[
	\begin{tikzcd}
		Y \ar{r}{\iota} \ar{d}{\pi} \ar{dr}{\sigma} & X_f\times X_g\times X \ar{d}{\pi_X}\\
		X_f\times X_g & X
	\end{tikzcd} 
\]
The cyclic group \(\mu_d\) acts on \(Y\) by scaling the second coordinate of the
fiber. We endow \(X_f\times X_g\) with the trivial action rendering the diagram
\(\mu_d\)-equivariant. 

Define a family of \textit{Fourier-Mukai functors}
\[
	\Xi_{i,j}\colon\mathcal{D}(X_f\times X_g)\to \mathcal{D}[X/\mu_d]
\]
using the kernel \(\iota_\ast\mathcal{O}_Y\otimes
\pi_X^\ast\mathcal{O}_X(iH)\otimes\chi^j\), i.e.
\[
	\Xi_{i,j}(\mathcal{F}^\cdot) = \mathbf{R}\pi_{X\ast}(\pi_{X_f\times
	X_g}^\ast(\mathcal{F}^\cdot)\otimes \iota_\ast
  \mathcal{O}_Y)\otimes \mathcal{O}_X(iH)\otimes \chi^j,
\]
where it is understood that before applying \(\Xi_{i,j}\) we precompose with the
embedding \(\mathcal{D}(X_f\times X_g)\hookrightarrow \mathcal{D}[X_f\times
X_g/\mu_d] = \oplus_{i=0}^{d-1}\mathcal{D}(X_f\times X_g)\otimes \chi^i\)
(into the trivial component). Then the derived push and pull functors are taken
equivariantly. 

We show \(\Xi_{i,j}\) is an embedding using Theorem \ref{thm:bondal-orlov-ff}.
Since the kernel is flat over \(X_f\times X_g\) we see
\(\Xi_{i,j}(\mathcal{O}_{(p,q)})\cong \mathcal{O}_{l(p,q)}(i)\otimes\chi^j\), where
\(l(p,q)\cong\mathbb{P}^1\) is the line in \(X\) joining \(\iota_f(p)\) to
\(\iota_g(q)\).

\begin{lemma}
	Let \(\mathcal{N}\) denote the normal bundle to \(l(p,q)\) inside \(X\). Then
	\[
		\mathcal{N}\cong (\oplus_{i=1}^{m-2}\mathcal{O}_{l(p,q)}(1))\oplus
		(\oplus_{j=1}^{n-2}\mathcal{O}_{l(p,q)}(1)\otimes\chi) \oplus
		\mathcal{O}_{l(p,q)}(2-d)\otimes \chi.
	\]
	\label{lem:normal-bundle-splitting}
\end{lemma}

\begin{proof}
	By the \(\mu_d\)-equivariant Grothendieck splitting theorem (Theorem
	\ref{thm:equivariant-grothendieck-splitting}), we have an isomorphism
	\[
		\mathcal{N}\cong
		\oplus_{i=1}^{m+n-3}\mathcal{O}_{l(p,q)}(n_i)\otimes\chi^{j_i}
	\]
	for some \(n_i\in\mathbb{Z}\) and weights \(j_i\).

	As \(X\) is a degree \(d\) hypersurface in \(\mathbb{P}^{m+n-1}\) and
	\(l(p,q)\) is a linear subvariety of \(\mathbb{P}^{m+n-1}\), the normal bundle
	\(\mathcal{N}\) fits into the following equivariant exact sequence:

	\[
		0\to \mathcal{N}\to
		(\oplus_{i=1}^{m-1}\mathcal{O}_{l(p,q)}(1))\oplus
		(\oplus_{j=1}^{n-1}\mathcal{O}_{l(p,q)}(1)\otimes\chi)\to
		\mathcal{O}_{l(p,q)}(d)\to 0
	\]
	on \(l(p,q)\). The weights come from the description of the morphism
	\[
		\mathcal{O}_{l(p,q)}(1)^{\oplus m+n-2} \to
		\mathcal{O}_{l(p,q)}(d).
	\]
	It is given by multiplication by
	\[
		(\partial_{u_1}f|_{l(p,q)}, \ldots, \partial_{u_{m-1}}f|_{l(p,q)}
		,\partial_{v_1}g|_{l(p,q)}, \ldots, \partial_{v_{n-1}}g|_{l(p,q)}), 
	\]
	where \( u_1,\ldots,u_{m-1}\) are linear sections cutting out
	\(p\in\mathbb{P}^{m-1}\) and \(v_1,\ldots,v_{n-1}\) are linear sections
	cutting out \(q\in\mathbb{P}^{n-1}\).
	
	Up to a linear change of coordinates, we can assume this mapping is
	\[
		(u_1^{d-1},0,\ldots,0,v_1^{d-1},0,\ldots,0).
	\]
	Hence,
	\[
		\mathcal{N}\cong
		(\oplus_{i=1}^{m-2}\mathcal{O}_{l(p,q)}(1))\oplus
		(\oplus_{j=1}^{n-2}\mathcal{O}_{l(p,q)}(1)\otimes\chi)\oplus
		\mathcal{O}(i)\otimes\chi^j
	\]
	Since \(\deg(\mathcal{N}) = m+n-2-d\) we must have \(i = 2-d\). By checking the
	stalks of the normal bundle exact sequence, we must have \(j = 1\). 
\end{proof}

\begin{lemma}
	For \( (p,q),(p',q')\in X_f\times X_g\). If \(p\neq p'\) or \(q\neq q'\), then
	\[
		\Ext^\ast_{[X/\mu_d]}(\mathcal{O}_{l(p,q)},\mathcal{O}_{l(p',q')})= 0.
	\]
	\label{lem:vanishing-distinct-points}
\end{lemma}

\begin{proof}
	If \(p\neq p'\) and \(q\neq q'\), then the subvarieties \(l(p,q)\) and
	\(l(p',q')\) are disjoint. The vanishing follows. 

  Suppose \(p = p'\) and \(q\neq q'\). We must compute
	\[
		\Ext^\ast_{[X/\mu_d]}(\mathcal{O}_{l(p,q)}, \mathcal{O}_{l(p,q')})\cong
		\Ext^\ast_{\mathcal{O}_{X,p}}(
		\mathcal{O}_{l(p,q),p},\mathcal{O}_{l(p,q'),p})^{\mu_d}
	\]

  We will use \ref{thm:cks-ss}. Set \(S = l(p,q)\) and \(T = l(p,q')\), then \(W
  = \{p\}\). We have \(\mathcal{N}_{W/T} = T_{l(p,q')}(p) \cong \chi\). Since
  also \(T_{l(p,q)}(p)\cong \chi\), we have by Lemma
  \ref{lem:normal-bundle-splitting}
  \[
    \tilde{N}|_p = (\mathcal{N}|_p)/\chi\cong k^{\oplus m-2}\oplus
    \chi^{\oplus n-2}.
  \]
  Here \(k\) is the trivial character. Since \(p\) is a point, we have
  \[
    \mathrm{Ext}^\ast_X(\mathcal{O}_{l(p,q)},\mathcal{O}_{l(p,q')})\cong
    H^0(\mathcal{E}xt^\ast_X(\mathcal{O}_{l(p,q)},\mathcal{O}_{l(p,q')}))\cong
    \chi\otimes\Lambda^{q-m}(k^{\oplus m-2}\oplus \chi^{\oplus n-2})
  \]
  and hence the weights of this extension group are always between \(1\) and
  \(n-1\). In particular, they are nontrivial and so we have the desired
  vanishing.

  The case \(p\neq p'\) and \(q=q'\) is analogous.
\end{proof}

We can now prove \(\Xi_{i,j}\) is fully-faithful.

\begin{theorem}
	The functors \(\Xi_{i,j}\) are fully-faithful for all \(i,j\).
	\label{thm:xi-qff}
\end{theorem}

\begin{proof} 
	Using Theorem \ref{thm:bondal-orlov-ff} and Lemma
	\ref{lem:vanishing-distinct-points} we only need to show
	\[
		\Ext^\ast_{[X/\mu_d]}(\mathcal{O}_{l(p,q)},\mathcal{O}_{l(p,q)}) =
		\begin{cases}
			k & \ast = 0 \\
			0 & \ast\notin [0,m+n-4]
		\end{cases}.
	\]
	That \(\Hom_{[X/\mu_d]}(\mathcal{O}_{l(p,q)},\mathcal{O}_{l(p,q)}) \cong k\) is
	clear, we now show vanishing.

	For this we use the local-to-global spectral sequence. Since \(l(p,q)\)
	and \(X\) are smooth, this reduces to:
	\[
		\mathrm{H}^r(\Lambda^s\mathcal{N})\Rightarrow
		\Ext^{r+s}_{[X/\mu_d]}(\mathcal{O}_{l(p,q)},\mathcal{O}_{l(p,q)}).
	\]
	here \(\mathcal{N}\) is the normal bundle from Lemma
	\ref{lem:normal-bundle-splitting}.

	To establish the relevant vanishing, we must compute
	\(\mathrm{H}^r(\Lambda^s\mathcal{N})\) for \( (r,s) = (0,m+n-3),
	(1,m+n-4)\). We will compute separately.
	
	For the case \( (r,s) = (0,m+n-3)\), we have
	\[
		\Lambda^{m+n-3}\mathcal{N}\cong
		\mathcal{O}_{l(p,q)}(m+n-4+2-d) \otimes \chi^{n-1}\cong
		\mathcal{O}(m+n-d-2)\otimes\chi^{n-1}.
	\]
	Suppose \(m+n-d-2\geq 0\) (otherwise there is nothing to check), then we would
	require a monomial of the form \(x^ay^{n-1}\) with \(a\geq 0\) and \(a+n-1 =
	m+n-d-2\). Solving for \(a\), we have \(a = m-d-1\leq -1\), since
	\(d\geq m\), which is impossible.

	Now for the case \( (r,s) = (1,m+n-4)\). By the decomposition of Lemma
  \ref{lem:normal-bundle-splitting}, only the summands
  of \(\Lambda^{m+n-4}\mathcal{N}\) that involve
  \(\mathcal{O}_{l(p,q)}(2-d)\otimes\chi\) as a tensor factor can contribute towards
  \(H^1(\Lambda^{m+n-4}\mathcal{N})\). In which case, the
	isotypical summands of \(\Lambda^{m+n-4}\mathcal{N}\) involving
	\(\mathcal{O}(2-d)\otimes\chi\) are:
	\[
    \mathcal{O}_{l(p,q)}(m+n-3-d)\otimes\chi^{n-2},
    \mathcal{O}_{l(p,q)}(m+n-3-d)\otimes\chi^{n-1}.
	\]

	If \(m+n-3-d\geq -1\), then the first cohomology group is zero even before
  taking invariants. Assume \(m+n-3-d\leq -2\). By Serre duality, we have
	isomorphisms:
  \begin{align*}
    H^1(\mathcal{O}_{l(p,q)}(m+n-3-d)\otimes\chi^{n-2})^{\mu_d}&\cong
    H^0(\mathcal{O}_{l(p,q)}(d+1-m-n)\otimes\chi^{1-n})^{\mu_d}; \\
    H^1(\mathcal{O}_{l(p,q)}(m+n-3-d)\otimes\chi^{n-1})^{\mu_d}&\cong
    H^0(\mathcal{O}_{l(p,q)}(d+1-m-n)\otimes\chi^{-n})^{\mu_d}.
  \end{align*}

	We remark that \(d>n\) here; otherwise, if \(d = n\), then \(m-3\leq -2\)
	forces \(m=1\) and we assume \(m\geq 2\). In particular, the weights
  \(\chi^{-n}\) and \(\chi^{1-n}\) are nontrivial above. 

  In the \(\chi^{-n}\) case, we must find a monomial of the form
	\(x^ay^{d-n}\) where \(a\geq 0\) and \(a+d-n = d+1-m-n\). This forces \(a =
	1-m\), which is absurd. Similarly, we would need a monomial of the form
	\(x^ay^{d-n+1}\) with \(a\geq 0\) and \(a+d-n+1 = d+1-m-n\) and hence \(a =
	-m\), which is still absurd. Thus there are no equivariant global sections and
  the group vanishes. The \(\chi^{1-n}\) case is similar.

	We conclude \(\Xi_{0,0}\) is fully-faithful and so \(\Xi_{i,j}\) is
	fully-faithful for all \(i,j\in\mathbb{Z}\) as it differs from \(\Xi_{0,0}\)
  by the autoequivalence \(\mathcal{F}^\cdot\mapsto \mathcal{F}^\cdot\otimes\mathcal{O}_X(iH)\otimes\chi^j\).
\end{proof}

Let \(\mathcal{D}_{fg} = \Xi_{-m,-n}\mathcal{D}(X_f\times X_g)\). By Lemma
\ref{thm:xi-qff} we have that \(\Xi_{-m,-n}\) is a full embedding. The main result can
now be stated.

\begin{main-theorem}
	In the above notation, we have a semi-orthogonal decomposition
	\[
		\mathcal{D}[X/\mu_d] = \langle
		\mathcal{D}_{g1}, \mathcal{D}_{fg}, \mathcal{D}_{g2},\mathcal{D}_f,
		\mathcal{A}\rangle.
	\]
	\label{thm:main-result}
\end{main-theorem}

The proof of this theorem will occupy \S \ref{sec:sod}-\ref{sec:full}. In \S
\ref{sec:sod} we finish proving that the decomposition is semi-orthogonal. We analyze
other sheaves that we can construct from the components in \S \ref{sec:koszul}.
They will be necessary in using other kernels and proving fullness. In \S
\ref{sec:full} we complete the proof of fullness. 

It is worth noting that in the cases \( (m,n) = (2,2), (2,3), (3,3)\), there is
an easier proof of this result. The idea of the proof is what is used in the
subsequent sections and so we believe it does no harm in proving these special
cases now.

\begin{theorem}
  The main theorem holds in the special cases 
  \[ 
    (m,n) = (2,2), (2,3), (3,3).
  \]
  \label{thm:mt-specialcase}
\end{theorem}

\begin{proof}
	We will only do the case \( (m,n) = (2,2)\) with the understanding that the
	other two are similar. The subcategories are as follows:
	\begin{align*}
		\mathcal{D}_{g1} &= \mathcal{D}_g = \langle \mathcal{D}_g^{2-d},\ldots,
		\mathcal{D}_g^{-1}\rangle, \\
		\mathcal{D}_{fg} &= \Xi_{-2,-2}(\mathcal{D}(X_f\times X_g)), \\
		\mathcal{D}_f &= \langle \mathcal{D}_f^{d-2},\ldots,\mathcal{D}_f^1\rangle,
		\\
		\mathcal{A} &= \langle \mathcal{O}_X(-2)\otimes\chi^{-1},
		\mathcal{O}_X(-1)\chi^{0,-1}, \mathcal{O}_X\rangle, \\
	\end{align*}
	Define \(\mathcal{T} = \langle \mathcal{D}_g, \mathcal{D}_{fg},
	\mathcal{D}_f, \mathcal{A}\rangle\). We will prove orthogonality in \S
	\ref{sec:sod}, the difficult part is fullness. To do this, it suffices to show
  \(\mathcal{T}\) has a spanning class; see Proposition
  \ref{prop:subcat-spanning-class}. 

	Using Example \ref{ex:spanning-class-stack}, we see that the collection of
	objects consisting of free orbits, say \(\mathcal{O}_Z\) where \(Z =
	\{\lambda\cdot z\}_{\lambda\in\mu_d}\) and \(\lambda\cdot z\neq z\), as well
	as the sheaves \(\mathcal{O}_{\iota_f(p)}\otimes\chi^i\) and
	\(\mathcal{O}_{\iota_g(q)}\otimes\chi^i\) for \(i = 1,\ldots,d\) form a
	spanning class. 

	Let \(J =J(X_f,X_g)\) inside of \(X\) denote the join of \(X_f\) and \(X_g\).
  A free orbit \(Z\subset X\setminus J\) is a complete intersection. To see
  this, take \([p:q]\in Z\) and let \(s_x\) be a linear form cutting out \(p\)
  in the variables \(x_1,x_2\) and similarly let \(s_y\) be a linear form cutting out
  \(q\) in the variables \(y_1,y_2\). For any other element \([p:\lambda q]\in Z\), we
  have \(s_x[p:\lambda q] = s_x(p) = 0\) and \(s_y[p:\lambda q] = s_y(\lambda q)
  = \lambda s_y(q) = 0\). Since \(Z\) is a free orbit (of size \(d\)) and \(X\) is degree \(d\),
  it follows that \(Z\) is a complete intersection. The corresponding resolution of \(\mathcal{O}_Z\) given by
	\[
		0\to \mathcal{O}_X(-2)\otimes\chi^{-1}\to \mathcal{O}_X(-1)\oplus
		\mathcal{O}_X(-1)\otimes\chi^{-1}\to \mathcal{O}_X
	\]
  is in the subcategory \(\mathcal{A}\). Hence
	\(\mathcal{O}_Z\in \mathcal{T}\) for every free orbit \(Z\subset X\setminus
  J\). 

	Now let \(l(p,q)\) denote the line joining
	\(\iota_f(p)\) to \(\iota_g(q)\). We will see in \S \ref{sec:koszul} that the
	objects \(\mathcal{O}_{l(p,q)}(-d)\) and \(\mathcal{O}_{l(p,q)}\) are in
	\(\mathcal{T}\) for all \(p,q\). This implies the structure sheaves of free
  orbits on \(l(p,q)\) are contained in \(\mathcal{T}\). It remains to see that the twists of the
	fixed orbits are in \(\mathcal{T}\).

	Using \(\mathcal{D}_g\) and \(\mathcal{D}_f\) we only need one additional
	twist, namely \(\mathcal{O}_p\), \(\mathcal{O}_q\) (or in the case (m,n) =
	(2,3),(3,3) we will also need
	\(\mathcal{O}_p\otimes\chi^{-1},\mathcal{O}_q\otimes\chi\)). To do that, we
	notice that \(\mathrm{Cone}(\mathcal{O}_X(-1)\to \mathcal{O}_X)\cong
	\mathcal{O}_{J(p,X_g)}\in \mathcal{A}\) by cutting out \(p\) with a section of
	\(\mathcal{O}_X(1)\). We then have the exact sequence
	\[
		0\to \mathcal{O}_{J(p,X_g)}\to \bigoplus_{q\in X_g}\mathcal{O}_{l(p,q)}\to
		\mathcal{O}_p^{\oplus d-1}\to 0.
	\]
	Since \(\mathcal{T}\) is saturated, it follows that
	\(\mathcal{O}_p\in\mathcal{T}\). In the case \( (m,n) = (2,3),(3,3)\) we can
	look at a similar sequence using
	\(\mathcal{O}_{J(p,X_g)}(-1)\otimes\chi^{-1}\in\mathcal{A}\). A similar
  argument shows \(\mathcal{O}_q\in\mathcal{T}\) and we are finished.
\end{proof}

\section{Semi-orthogonality}
\label{sec:sod}

\subsection{Geometric subcategories.}

\subsubsection{The subcategory \(\mathcal{D}_{fg}\).} 

As before, we let \(\mathcal{D}_{fg}\) be the image of the fully-faithful
functor \(\Xi_{-m,-n}\).  Let us compute the semi-orthogonality \(\langle
\mathcal{D}_{fg},\mathcal{A}\rangle\). We have the formula
\begin{align*}
	\Ext^\ast_{[X/\mu_d]}(\mathcal{O}_X(-i)\otimes\chi^{-j},\mathcal{O}_{l(p,q)}(-m)
	\otimes\chi^{-n})
  \cong H^\ast(\mathcal{O}_{l(p,q)}(i-m)\otimes\chi^{j-n}).
\end{align*}

\begin{lemma}
	There is a semi-orthogonal decomposition \(\langle
	\mathcal{D}_{fg},\mathcal{A}\rangle\). 
	\label{lem:sod-dfg-a}
\end{lemma}

\begin{proof}
	We only check the semi-orthogonality \(\langle
	\mathcal{D}_{fg},\mathcal{A}_3\rangle\) as the other computations are similar.
	The objects in \(\mathcal{A}_3\) are of the form
	\(\mathcal{O}(-i)\otimes\chi^{-j}\), where \(0\leq i\leq m-1\) and \(0\leq
	j\leq i\). In this case we have \(\mathcal{O}(i-m)\otimes\chi^{j-n}\) is a negative
	line bundle so
	\[
		H^1(\mathcal{O}(i-m)\otimes\chi^{j-n})\cong
		H^0(\mathcal{O}(m-i-2)\otimes\chi^{n-j-1})
	\]

	Since \(i\geq j\geq 0\) we have
	\[
		n-1\geq n-j-1\geq n-i-1
	\]
	and so we need a monomial of the form \(x^ay^{n-j-1}\) where \(a\geq 0\) and
	\(a+n-j-1 = m-i-2\). However
	\[
		a+n-j-1\geq n-i-1\geq m-i-1>m-i-2.
	\]
  which is impossible.
\end{proof}

\subsubsection{\(\mathcal{D}_f\) and \(\mathcal{D}_g\).} 

For \(\mathcal{D}_f\) to be present in the semi-orthogonal decomposition, we
need \(d>n\) and for \(\mathcal{D}_g\) to be present, we require \(d>m\).

\begin{lemma}
	We have the semi-orthogonality \(\langle \mathcal{D}_g,\mathcal{D}_f,\mathcal{A}\rangle\).
	\label{lem:sod-df-a}
\end{lemma}

\begin{proof}
	The subcategories \(\mathcal{D}_g\) and \(\mathcal{D}_f\) are orthogonal. We
	only show \(\mathcal{D}_f\) is right orthogonal to \(\mathcal{A}\), the claim
	that \(\mathcal{D}_g\) is also right orthogonal is analgous. 
	
  Let \(\mathcal{O}(-i_1)\otimes\chi^{-i_2}\in\mathcal{A}\) and
  \(\mathcal{O}_p\otimes \chi^{-j}\in D_f\). By Proposition
  \ref{prop:sod-spanning}, it suffices to show
  \(\mathrm{Ext}^\ast_X(\mathcal{O}(-i_1)\otimes
  \chi^{-i_2},\mathcal{O}_p\otimes\chi^{-j})^{\mu_d} = 0\) for each such pair. We
  compute
	\[ 
		\Ext^\ast_{[X/\mu_d]}(\mathcal{O}(-i_1)\otimes\chi^{-i_2},\mathcal{O}_p\otimes\chi^{-j})
    \cong \Gamma(\mathcal{O}_p\otimes\chi^{i_2-j})^{\mu_d}.
	\]
  By definition of \(\mathcal{A}\) and \(\mathcal{D}_f\), we have \(0\leq
  i_2\leq n-1\) and \(n\leq j\leq d-1\). Hence, \(i_2-j\) is not divisible by
  \(d\) which means that the invariants of \(\mathcal{O}_p\otimes\chi^{i_2-j}\)
  vanish.
\end{proof}

\begin{lemma}
	We have the semi-orthogonality \(\langle
	\mathcal{D}_{g1},\mathcal{D}_{fg},\mathcal{D}_{g2},\mathcal{D}_f\rangle\).
	\label{lem:sod-dfg-df}
\end{lemma}

\begin{proof}
	Again, we only prove the semi-orthogonality \( \langle \mathcal{D}_{fg},
	\mathcal{D}_f\rangle\) the other claims are analagous.  The only possible
	nonzero extension groups between members of the standard spanning classes of
  \(\mathcal{D}_f\) and \(\mathcal{D}_{fg}\) are
	\[
		\Ext^\ast_{[X/\mu_d]}(\mathcal{O}_p
    \otimes\chi^{-j},\mathcal{O}_{l(p,q)}(-m)\otimes\chi^{-n}) \cong
    \left(\Ext^\ast_X(\mathcal{O}_p,\mathcal{O}_{l(p,q)})\otimes\chi^{j-n}\right)^{\mu_d},
	\]
  where \(l(p,q)\) is the line between \(p\in X_f\) and and any point \(q\in X_g\). 

  We use Theorem \ref{thm:cks-ss} with \(S = \{p\}, T= l(p,q), W= \{p\}\). In
  this case, we have
  \begin{align*}
    \Ext^q_X(\mathcal{O}_p,\mathcal{O}_{l(p,q)}) \otimes\chi^{j-n}\cong
    \Lambda^{q-1}(k^{\oplus m-2}\oplus \chi^{\oplus n-1})\otimes\chi^{j-n+1}
  \end{align*}
  Since \(n\leq j\leq d-1\), the weights on the extension groups will be between
  \(1\) and \(d-1\). Thus there are no invariants.
\end{proof}

This completes the semi-orthogonality claim in the Main Theorem. It
remains to see fullness.

\section{Koszul Complexes of Joins and Orbits.}
\label{sec:koszul}

\subsection{Koszul complexes}
Let \(\mathcal{T} = \langle \mathcal{D}_{g1},\mathcal{D}_{fg},\mathcal{D}_{g2},
\mathcal{D}_f,\mathcal{A}\rangle\). We prove fullness by showing \(\mathcal{T}\)
has a spanning class. This is done in Section \ref{sec:full}. To do so, we need
to construct more sheaves in \(\mathcal{T}\) using the subcategories present.
Essential to these constructions is the Koszul complex of a regular section of a
vector bundle.

Let \(E\) be a \(\mu_d\)-equivariant locally free sheaf of rank \(r\) over \(X\)
and \(s\in \Gamma(E)^{\mu_d}\) be an equivariant global section. Then we have
the corresponding Koszul complex:
\[
	0\to \Lambda^rE^\vee\to \Lambda^{r-1}E^\vee\to \cdots \to
	E^\vee\xrightarrow{s^\vee} \mathcal{O}_X.
\]
Denote this complex by \(\mathcal{K}(E,s)\).

If the zero locus \(Z(s)\subset X\) of \(s\) is codimension \(r\), then the
Koszul complex is exact and is a locally free resolution of
\(\mathcal{O}_{Z(s)}\). Even if the Koszul complex is not exact we can still
learn information from its cohomology sheaves, see Section
\ref{ssec:lines-in-join}.

\subsection{Free orbits away from \(J(X_f,X_g)\)}
\label{ssec:free-orbits}

Let \(Z\subset X\) be a free orbit of the \(\mu_d\) action away from the join of
\(X_f\) and \(X_g\) inside \(X\), i.e.\ pick \([p:q]\notin X\setminus
J(X_f,X_g)\) and let \(Z = \{[p:\lambda q]\mid \lambda\in\mu_d\}\). Let
\(s_1,\ldots,s_{m-1}\) be linear sections in the variables \(x_1,\ldots,x_m\) that
cut out \(p\) and \(t_1,\ldots,t_{n-1}\) be linear sections in the variables
\(y_1,\ldots,y_n\) that cut out \(q\). Then it follows, as in the proof of
Theorem \ref{thm:mt-specialcase}, \(Z = V(s_1,\ldots,s_{m-1},t_1,\ldots,t_{n-1})\subset X\).
Hence, \(Z\) is the vanishing locus of the section
\(s=(s_1,\ldots,s_{m-1},t_1,\ldots,t_{n-1})\in\Gamma(E)\), where \(E = \left(
  \bigoplus_{i=1}^{m-1}\mathcal{O}_X(1)\right)\oplus\left(
\bigoplus_{j=1}^{n-1}\mathcal{O}_X(1)\otimes\chi  \right)\). 

The summands of \(\mathcal{K}(E,s)\) are precisely the sheaves that occur as
generators of \(\mathcal{A}\). We conclude for any free orbit \(Z\) in \(X\setminus
J(X_f,X_g)\), we have \(\mathcal{O}_Z\in\mathcal{A}\).

\subsection{Lines in \(J(X_f,X_g)\)}
\label{ssec:lines-in-join}

Let \(p\in X_f\) and \(q\in X_g\). As before, denote by \(l(p,q)\cong
\mathbb{P}^1\) the line joining \(\iota_f(p)\) to \(\iota_g(q)\). Contrary to
the case of free orbits outside of \(J(X_f,X_g)\), the structure sheaves of both
fixed orbits and free orbits in the join are not complete intersections.
Moreover, the structure sheaf \(\mathcal{O}_{l(p,q)}\) is not a complete
intersection subvariety.  We can still consider a Koszul complex using the same
construction of the section \(s = (s_1,\ldots,s_{m-1},t_1,\ldots,t_{n-1})\in
\Gamma(E)\), where \(E\) is as before, such that \(V(s) = l(p,q)\).

\begin{lemma}
	As above, let \(\mathcal{K}(E,s)\) be the Koszul complex associated to the
  section \(s\in\Gamma(E)\). Then
	\[
		\mathcal{H}^\ast(\mathcal{K}(E,s))= \begin{cases} 
			\mathcal{O}_{l(p,q)} & \ast = 0 \\ 
			\mathcal{O}_{l(p,q)}(-d) & \ast = -1 \\
			0 & \ast\neq 0,-1
		\end{cases}
	\]
	\label{lem:koszul-lines}
\end{lemma}

\begin{proof}
  Since \(X\) is a closed embedding, pushforward is exact so we may compute
  cohomology sheaves on \(\mathbb{P}^{m+n-1}\). To that end, consider the
  analogous Koszul complex \(\mathcal{K}(\mathbb{P}^{m+n-1},E,s)\) on
  \(\mathbb{P}^{m+n-1}\). There is a short exact sequence of complexes
  \[
    0\to \mathcal{K}(\mathbb{P}^{m+n-1},E,s)(-d)\to
    \mathcal{K}(\mathbb{P}^{m+n-1},E,s)\to \mathcal{K}(E,s)\to 0
  \]
  on \(\mathbb{P}^{m+n-1}\). Since \(l(p,q)\) is a complete intersection in
  \(\mathbb{P}^{m+n-1}\), the Koszul complex is a locally free resolution and
  hence we have an isomorphism \(\mathcal{K}(\mathbb{P}^{m+n-1},E,s)\cong
  \mathcal{O}_{l(p,q)}\). Since the mapping
  \[
    \mathcal{H}^0(\mathcal{K}(\mathbb{P}^{m+n-1},E,s))\to
    \mathcal{H}^0(\mathcal{K}(E,s))
  \]
  is restriction and both sheaves are isomorphic to \(\mathcal{O}_{l(p,q)}\), it
  is an isomorphism. Thus the long exact sequence of
  cohomology objects yields the additional isomorphism
  \[
    \mathcal{H}^{-1}(\mathcal{K}(E,s))\cong
    \mathcal{H}^0(\mathcal{K}(\mathbb{P}^{m+n-1},E,s)(-d))\cong
    \mathcal{O}_{l(p,q)}(-d).
  \]
\end{proof}

\subsection{Projective cones}

The subvariety \(\iota_g\colon X_g\to X\) is a complete intersection. Indeed, it is the
zero locus of the section \(s_{X_g} = (x_1,\ldots,x_m)\) of the vector bundle
\(E_{X_g}
= \mathcal{O}_X(1)^{\oplus m}\). The summands of the Koszul resolution,
\(\mathcal{K}(E_{X_g},s_{X_g})\) are of the form \(\mathcal{O}_X(-m+i)\) for \(i
= 0,\ldots,m\).

\begin{lemma}
  For \(i=1,\ldots,d-m\), the components of
	\(\mathcal{K}(E_{X_g},s_{X_g})(-(n-1)+i+t)\otimes\chi^{-(n-1)+t}\) are in
	\(\mathcal{T}\) for \(t=0,\ldots,n-1\).
	\label{lem:missing-line-bundles}
\end{lemma}

\begin{proof}
	We check the base case \(i =1\). In this case, we have explicitly:
	\begin{align*}
		\mathcal{K}(E_{X_g},s_{X_g})(1) &\to \mathcal{O}_{X_g}(1)\otimes\chi^{-1} \\
		\mathcal{K}(E_{X_g},s_{X_g})\otimes\chi^{-1} &\to
		\mathcal{O}_{X_g}\otimes\chi^{-1} \\
		\vdots &\to \vdots \\
		\mathcal{K}(E_{X_g},s_{X_g})(-(n-1)+1)\otimes\chi^{-(n-1)} &\to
		\mathcal{O}_{X_g}(-(n-1)+1)\otimes\chi^{-1}
	\end{align*}
	All line bundles appearing in the resolution are already in \(\mathcal{T}\)
	except the line bundle appearing in degree zero. 
	Since \(\mathcal{O}_{X_g}(j)\otimes\chi^{-1}\in \mathcal{T}\) for all \(j\),
	we know that the line bundles in degree zero are also in \(\mathcal{T}\).

  Suppose the assertion is true for \(1,\ldots,i\) we show that it is true for
  \(i+1\leq d-m\). In which case we have the following twists of the above
  diagram:
	\begin{align*}
		\mathcal{K}(E_{X_g},s_{X_g})(i+1) &\to
		\mathcal{O}_{X_g}(i+1)\otimes\chi^{-i-1} \\
		\mathcal{K}(E_{X_g},s_{X_g})(i)\otimes\chi^{-1} &\to
		\mathcal{O}_{X_g}(i)\otimes\chi^{-i-1} \\
		\vdots &\to \vdots \\
		\mathcal{K}(E_{X_g},s_{X_g})(-(n-1)+i+1)\otimes\chi^{-(n-1)} &\to
		\mathcal{O}_{X_g}(-(n-1)+i+1)\otimes\chi^{-i-1}.
	\end{align*}
	Again, all of the sheaves except the rightmost part of the resolution are
	already in \(\mathcal{T}\) by induction. That the line bundles in degree zero
	are in \(\mathcal{T}\) follows since
	\(\mathcal{O}_{X_g}(j)\otimes\chi^{-i-1}\in \mathcal{T}\) as \(i+1\leq d-m\).
\end{proof}

\begin{lemma}
	Let \(J(X_f,q)\) denote the join of \(X_f\) and \(q\) inside \(X\). Then
	\[
		\mathcal{O}_{J(X_f,q)}(i)\in\mathcal{T}
	\]
	for \(i = 0,\ldots,d-m\).
	\label{lem:join-xf-q}
\end{lemma}

\begin{proof}
	The subvariety \(J(X_f,q)\) is a complete interesection with Koszul resolution:
	\[
		(\otimes_{i=1}^{n-1} \mathcal{O}_X(-1)\otimes\chi^{-1} \to
		\mathcal{O}_X)\to \mathcal{O}_{J(X_f,q)}.
	\]
  After twisting by \(\mathcal{O}_X(i)\), all summands of components of the
  Koszul resolution are of the form 
	\[
		\mathcal{O}_X(-(n-1)+i+t)\otimes\chi^{-(n-1)+t}.
	\]
	The statement now follows from Lemma \ref{lem:missing-line-bundles}.
\end{proof}

As \(X_f\) is also a complete intersection subvariety, we have the following similar
statement for \(\mathcal{O}_{J(p,X_g)}\) with \(p\in X_f\).

\begin{lemma}
	Let \(\mathcal{O}_{J(p,X_g)}\) denote the join of \(p\in X_f\) with \(X_g\).
	Then
	\[
		\mathcal{O}_{J(p,X_g)}(i)\otimes\chi^i\in\mathcal{T}
	\]
	for \(i = 0,\ldots,d-n\).
	\label{lem:join-xg-p}
\end{lemma}

\section{Fullness.}
\label{sec:full}

Using Proposition \ref{prop:subcat-spanning-class}, our goal is to show that
\(\mathcal{T}=\langle
\mathcal{D}_{g1},\mathcal{D}_{fg},\mathcal{D}_{g2},\mathcal{D}_f,\mathcal{A}\rangle\)
contains a spanning class of \(\mathcal{D}[X/\mu_d]\). Recall, Example
\ref{ex:spanning-class-stack}, if \(\mathcal{X}\) is a smooth DM stack with coarse
moduli space, the sheaves
\[
  \Omega = \{\mathcal{O}_{x,\xi}\mid x\colon\mathrm{Spec}(k)\to
  \mathcal{X}\text{ and }\xi\in\mathrm{Irr}(\mathrm{Aut}(x))\}
\]
form a spanning class.

For \(\mathcal{X} = [X/\mu_d]\), these sheaves are the structure sheaves of the
free orbits and twists of the structure sheaves of fixed orbits by all
characters. In \S \ref{ssec:free-orbits} we saw that the structure sheaves of
the free orbits away from \(J(X_f,X_g)\) have Koszul resolutions using the
sheaves in \(\mathcal{A}\). We will get the remaining sheaves by showing for all
\(p\in X_f\) and \(q\in X_g\) the structure sheaves of free orbits on \(l(p,q)\)
and all twists of the structure sheaves of \(\mathcal{O}_p\) and
\(\mathcal{O}_q\) are in \(\mathcal{T}\).

\subsection{Other kernels.}

It will be convenient to use the images of other Fourier-Mukai kernels from
\(\Xi_{-m,-n}\) to \(\Xi_{d-m,0}\) and \(\Xi_{d-n,d-n}\). We justify their use
in this subsection.

Using Theorem \ref{thm:swapping-functors} we must verify that the image of the
spanning class \(\{\mathcal{O}_{(p,q)}\}\) under \(\Xi_{d-m,0}\) and
\(\Xi_{d-n,d-n}\) lies in \(\mathcal{T}\).  We start by
verifying the free orbits along the join are contained in
\(\mathcal{T}\).

\begin{lemma}
	For all \(p\in X_f\) and \(q\in X_g\) we have \(\mathcal{O}_{l(p,q)}(-d)\) and
	\(\mathcal{O}_{l(p,q)}\) in \(\mathcal{T}\).
	\label{lem:missing-p1-lines}
\end{lemma}

\begin{proof}
	By Lemma \ref{lem:koszul-lines}, it suffices to show
	\(\mathcal{O}_{l(p,q)}(-d)\in\mathcal{T}\). We have the exact sequences
	\[
		0\to \mathcal{O}_{l(p,q)}(-m-i-1)\otimes\chi^{-n} \to
		\mathcal{O}_{l(p,q)}(-m-i)\otimes \chi^{-n} \to
		\mathcal{O}_q\otimes\chi^{-(n-m)+i}\to 0.
	\]
	Since \(\mathcal{O}_{l(p,q)}(-m)\otimes\chi^{-n},
	\mathcal{O}_q\otimes\chi^{-(n-m)},\ldots,\mathcal{O}_q\otimes\chi^{-1} \in
	\mathcal{T}\), we have
	\(\mathcal{O}_{l(p,q)}(-n)\otimes\chi^{-n}\in\mathcal{T}\) by induction.

	Now consider the sequences
	\[
		0\to \mathcal{O}_{l(p,q)}(-n-i-1)\otimes\chi^{-n-i-1} \to
		\mathcal{O}_{l(p,q)}(-n-i)\otimes\chi^{-n-i} \to \mathcal{O}_p
		\otimes\chi^{-n-i}\to 0.
	\]
	Since \(\mathcal{O}_{l(p,q)}(-n)\otimes \chi^{-n},
	\mathcal{O}_p\otimes\chi^1,\ldots,
	\mathcal{O}_p\otimes\chi^{d-n}\in\mathcal{T}\), we have
	\(\mathcal{O}_{l(p,q)}(-d)\in\mathcal{T}\) by induction and this completes the
	proof.
\end{proof}

\begin{corollary}
  For every free orbit \(Z\) on the join \(J(X_f,X_g)\), we have
  \(\mathcal{O}_Z\in\mathcal{T}\).
  \label{cor:free-orbit-join}
\end{corollary}

\begin{proof}
  Any free orbit \(Z\) on the join is contained in \(l(p,q)\) for some \(p\in X_f\) and \(q\in
  X_g\). Moreover, the structure sheaf fits in the short exact sequence
  \[
    0\to\mathcal{O}_{l(p,q)}(-d)\to\mathcal{O}_{l(p,q)}\to\mathcal{O}_Z\to 0
  \]
  and so \(\mathcal{O}_Z\in\mathcal{T}\).
\end{proof}

\begin{lemma}
	For all \(p\in X_f\) and \(q\in X_g\) we have \(\mathcal{O}_{l(p,q)}(d-m),
	\mathcal{O}_{l(p,q)}(d-n)\otimes\chi^{d-n}\in\mathcal{T}\).
	\label{lem:ample-kernels}
\end{lemma}

\begin{proof}
	Use the exact sequences in the proof of Lemma \ref{lem:missing-p1-lines}.
\end{proof}

\begin{proposition}
	The images of the functors \(\Xi_{d-m,0}\) and \(\Xi_{d-n,d-n}\) are contained
  in \(\mathcal{T}\).
  \label{prop:images}
\end{proposition}

\begin{proof}
	By Proposition \ref{prop:sod-saturatedness}, it follows from Lemma
	\ref{lem:ample-kernels} and Theorem \ref{thm:swapping-functors}.
\end{proof}

\subsection{Proof of Fullness}

\subsubsection{Strategy.}
We now compute \(\Xi_{d-m,0}(\mathcal{O}_{X_f\times\{q\}}(-i))\) and
\(\Xi_{d-n,d-n}(\mathcal{O}_{\{p\}\times X_g}(-j))\) for various \(i=0,\ldots,m-1\) and
\(j=0,\ldots,n-1\) to show that this gives us the missing sheaves:
\(\mathcal{O}_q\otimes\chi^{0,1,\ldots,m-1},
\mathcal{O}_p\otimes\chi^{0,-1,\ldots,-(n-1)}\). In particular, we show there
exists triangles
\[
	\mathcal{O}_{J(X_f,q)}(d-m-i)\to \Xi_{d-m,0}(\mathcal{O}_{X_f\times\{q\}}
	(-i))\to T_q\to
\]
and
\[
	\mathcal{O}_{J(p,X_g)}(d-n-j)\otimes\chi^{d-n} \to
	\Xi_{d-n,d-n}(\mathcal{O}_{\{p\}\times X_g})\to T_p\to 
\]
for every \(p\in X_f\) and \(q\in X_g\), where \(T_q\) and \(T_p\) are certain
torsion sheaves supported at \(q\) and \(p\), respectively. Since the first two
objects are in \(\mathcal{T}\) we have \(T_q,T_p\in \mathcal{T}\). We then
build a filtration of \(T_q,T_p\) and argue then that \(\mathcal{T}\) contains
the remaining elements of the spanning class.

\subsubsection{Factoring through the blowup.}

Take \(q\in X_g\) and let \(\mathbb{P}^m\) be the linear subspace spanned by
\(x_1,\ldots,x_m,q\) or equivalently as the join of \(\mathbb{P}^{m-1}\) and
\(q\) inside \(\mathbb{P}^{m+n-1}\). Consider the following commutative diagram
\[
	\begin{tikzcd}
		\mathbb{P}(\mathcal{O}_{X_f}(-1)\boxplus
		\mathcal{O}_{X_g}(-1))=Y \ar{d} & \mathbb{P}(\mathcal{O}_{X_f}(-1)\oplus\mathcal{O}_{X_f})
		= Z \ar{r}{\iota_q} \ar{d}{\pi} \ar{dr}{\sigma} \ar{l}&
		\Bl=\Bl_q\mathbb{P}^m\ar{d}{\beta} \\
		X_f\times X_g & \ar{l}{j}X_f & \mathbb{P}^m
	\end{tikzcd}
\]
where \(j\) includes \(X_f\) via \(j(x) = (x,q)\), \(Z\) is the fibered product,
and \(\sigma\) also denotes the restriction of \(\sigma\colon Y\to
\mathbb{P}^{m+n-1}\) to \(Z\) which factors through \(\mathrm{Bl}\). Notice
\(\sigma(Z)\) is a degree \(d\) hypersurface in \(\mathbb{P}^m\) as it is the
projective cone over \(X_f\). The mapping
\(\iota_q\) includes \(Z\) as the strict transform of \(d\) times a hyperplane divisor
\(H\) in \(\mathbb{P}^m\) along \(\beta\). Thus we have an equivalence
\(\iota_q(Z)\equiv dH-dE\) in \(\Bl\). Here
\(\mathbb{P}^m\) has coordinates \([x_1:\cdots:x_m:y]\) and \(\mu_d\) acts by
scaling the \(y\) coordinate. The \(\mu_d\)-action lifts to \(\Bl\) fixing the
exceptional divisor pointwise and thus rendering the entire diagram \(\mu_d\)-equivariant.

Recall, in \S \ref{ssec:serre} the canonical bundle of
\([\mathbb{P}^m/\mu_d]\) was computed to be
\(\mathcal{O}_{\mathbb{P}^m}(-m-1)\otimes\chi^{-1}\). The usual formula for the canonical bundle of a
blowup yields \(\omega_{\Bl}\cong \beta^\ast\mathcal{O}_{\mathbb{P}^m}\otimes
\mathcal{O}((m-1)E)\) which admits a \(\mu_d\)-linearization since the divisors
involved are invariant under the \(\mu_d\)-action. It remains to determine if
there is a twist by a character. Restricting to \(\Bl\setminus E\) gives the
isomorphism
\[
	\omega_{\Bl}|_{\Bl\setminus E}\cong
	\beta^\ast\omega_{\mathbb{P}^m}|_{\Bl\setminus E}
\]
and so it follows \(\omega_{\Bl}\cong \beta^\ast\mathcal{O}_{\mathbb{P}^m}(-m-1)
\otimes \mathcal{O}_{\Bl}((m-1)E)\otimes\chi^{-1}\).

Let \(H_1\) be a hyperplane section of \(X_f\) and \(H\) be a hyperplane section
of \(\mathbb{P}^m\) which restricts to \(H_1\) under the inclusion
\(X_f\to\mathbb{P}^m\) where \([x]\mapsto [x:0]\). Then \(H-E|_Z\cong \pi^\ast
H_1\). 

\subsubsection{Equivariant Grothendieck Duality.}

\begin{theorem}
	There is a natural isomorphism:
	\[
		\mathbf{R}\beta_\ast\mathcal{O}_{[\Bl/\mu_d]}(D)\cong
		\mathbf{R}\mathcal{H}om_{[\mathbb{P}^m/\mu_d]}
    (\mathbf{R}\beta_\ast\mathcal{O}(-D+(m-1)E),\mathcal{O})
	\]
	for any \(\mu_d\)-invariant divisor \(D\) on \(\mathrm{Bl}\).
	\label{thm:equivariant-grothendieck-duality}
\end{theorem}
 
\begin{proof}
	This follows since \(\beta\) is \(\mu_d\)-equivariant and the usual
	Grothendieck duality is natural, hence commutes with automorphisms.
\end{proof}

The divisors on \(\Bl\) are, up to equivalence, well known to be of the form
\(aH + bE\) for \(a,b\in\mathbb{Z}\). Using the projection formula, we have
\[
  \mathbf{R}\beta_\ast\omega_{\Bl}(-(aH+bE))\cong
	(R\beta_\ast\mathcal{O}_{[\Bl/\mu_d]}( (m-1-b)E))\otimes
	\omega_{[\mathbb{P}^m/\mu_d]}(-aH)
\]
and by Grothendieck duality
\[
	\mathbf{R}\beta_\ast(\mathcal{O}_{[\Bl/\mu_d]}(aH+bE))\cong
	\mathbf{R}\mathcal{H}om_{[\mathbb{P}^m/\mu_d]}
	(\mathbf{R}\beta_\ast(\mathcal{O}_{[\Bl/\mu_d]}(
	(m-1-b)E),\mathcal{O}_{\mathbb{P}^m})(-aH).
\]

\begin{remark}
  Since \(\{q\}\subset\mathbb{P}^m\) is of codimension \(m\), there is a canonical isomorphism
	\[
		\mathbf{R}\beta_\ast\mathcal{O}_{[\Bl/\mu_d]}(kE)\cong
    \mathcal{O}_{[\mathbb{P}^m/\mu_d]}
	\] 
  for \(k = 0,\ldots,m-1\). For \(k\geq m\), we still have
  \(\beta_\ast\mathcal{O}_{\mathrm{Bl}}(kE)\cong\mathcal{O}_{[\mathbb{P}^m/\mu_d]}\)
  but \(\mathbf{R}^{m-1}\beta_\ast\mathcal{O}_{\mathrm{Bl}}(kE)\) is non-zero.

  If \(k>0\), then \(\mathbf{R}\beta_\ast\mathcal{O}_{[\Bl/\mu_d]}(-kE)\cong
	\mathcal{I}_q^k\), where \(\mathcal{I}_q\) is the ideal sheaf for the closed
	subscheme \(\{q\}\) in \(\mathbb{P}^m\). Moreover,
	\(\mathbf{R}^i\beta_\ast\mathcal{O}_{[\Bl/\mu_d]}(kE) = 0\) unless \(i =
	0,m-1\).
	\label{rem:higher-image-exceptional}
\end{remark}

\subsubsection{Computing direct images.}

Let \(H_1\) be a hyperplane divisor in \(X_f\). For \(i = 0,\ldots,m-1\), we consider 
\(\Xi_{d-m,0}(\mathcal{O}_{X_f}(-iH_1))\). On \(\Bl\) we have the divisor exact
sequence
\[
	0\to \mathcal{O}_{[\Bl/\mu_d]}(-dH +dE)\to \mathcal{O}_{[\Bl/\mu_d]}\to
	\iota_{q\ast}\mathcal{O}_Z\to 0.
\]
Since the restriction of the hyperplane divisor \(H\) to \(X_f\) is \(H_1\), we
have isomorphisms 
\begin{align*}
  \iota_{q\ast}\pi^\ast\mathcal{O}_{X_f}(-iH_1)&=
  \iota_{q\ast}\mathcal{O}_Z(-iH_1) \\
  &\cong \iota_{q\ast}\sigma^\ast\mathcal{O}_{\mathbb{P}^m}(-iH) \\
  &\cong \iota_{q\ast}\iota_q^\ast\beta^\ast\mathcal{O}_{\mathbb{P}^m}(-iH) \\
  &\cong \iota_{q\ast}\mathcal{O}_Z\otimes\mathcal{O}_{[\Bl/\mu_d]}(-i(H-E)).
\end{align*}
Consider the twist of the divisor exact sequence
\[
	0\to \mathcal{O}_{[\Bl/\mu_d]}(-dH+dE-iH+iE)\to
	\mathcal{O}_{[\Bl/\mu_d]}(-iH+iE) \to \mathcal{O}_Z(-iH_1)\to 0.
\]
Using the  long exact sequence of cohomology sheaves for
\(\mathbf{R}\beta_\ast\), we see that, for \(k>0\), there is an isomorphism
\(\mathbf{R}^k\sigma_\ast\mathcal{O}_Z(-iH_1)\cong
\mathbf{R}^{k+1}\beta_\ast \mathcal{O}_{[\Bl/\mu_d]}( -(d+i)H + (d+i)E)\). It
follows from Remark \ref{rem:higher-image-exceptional}, that the only possible nonzero
higher direct image is \(\mathbf{R}^{m-2}\sigma_\ast \mathcal{O}_Z(-iH_1)\).

\begin{lemma}
	If \(m>2\), then for \(i =0,\ldots,m-1\), we have a distinguished triangle
	\[
		\mathcal{O}_{J(X_f,q)}(d-m-i)\to \Xi_{d-m,0}(\mathcal{O}_{X_f}(-iH_1))\to
		\mathcal{H}^{m-1}( (\mathcal{I}_q^{d-m+1+i})^\vee)(-m-i)[2-m]\to
	\]
	where \(\mathcal{I}_q\) is the ideal sheaf of \(\{q\}\) in \(\mathbb{P}^m\),
  \( (\mathcal{I}_q^{d-m+1+i})^\vee\) is the derived dual, and we have
  identified \(X_f\times\{q\}\) with \(X_f\). In particular, 
	\[
		\mathcal{H}^\ast(\Xi_{d-m,0}(\mathcal{O}_{X_f}(-iH_1)))\cong
		\begin{cases}
			\mathcal{O}_{J(X_f,q)}(d-m-i) & \ast = 0 \\
			\mathcal{H}^{m-1}( (\mathcal{I}_q^{d-m+1+i})^\vee)(-m-i) & \ast = m-2 \\
			0 & \ast \neq 0,m-2
		\end{cases}.
	\]
	If \(m=2\), then for \(i=0,1\) there is an exact sequence
	\[
		0\to \mathcal{O}_{J(X_f,q)}(d-2-i)\to \Xi_{d-2,0}(\mathcal{O}_{X_f}(-iH_1))\to
		\mathcal{O}_q^{\oplus d-1}\otimes\chi^{2+i-d}\to 0.
	\]
	\label{lem:cohomology-sheaves-functor}
\end{lemma}

\begin{proof}
	Notice that
  \begin{align*}
    \Xi_{d-m}(\mathcal{O}_{X_f}(-iH_1))&\cong
    (\sigma_\ast\pi^\ast(-iH_1)\otimes\mathcal{O}_{\mathbb{P}^m}((d-m)H)) \\
    &\cong
  \beta_\ast\iota_{q\ast}\mathcal{O}_Z(-iH_1)\otimes\mathcal{O}_{\mathbb{P}^m}((d-m)H).
  \end{align*}
  The case \(m=2\) is easy to see directly and the vanishing statements for
	\(m>2\) follow from the preceeding discussion. It remains to
	show the isomorphisms. Since \(d+i>0\), the only possible non-vanishing
  cohomology sheaf of \(\mathbf{R}\beta_\ast\)
  is in degree \(m-1\). By applying \(\beta_\ast\) to the divisor exact sequence
  for \(Z\subset\mathrm{Bl}\), we have an exact sequence
	\[
		0\to \mathcal{O}_{[\mathbb{P}^m/\mu_d]}(-d-i)\to
		\mathcal{O}_{[\mathbb{P}^m/\mu_d]}(-i)\to \mathcal{O}_{J(X_f,q)}(-i)\to 0.
	\]
	Now \(\mathcal{H}^0(\Xi_{d-m,0}(\mathcal{O}_{X_f}(-iH_1)))\cong
	\mathcal{H}^0(\Xi_{0,0}(\mathcal{O}_{X_f}(-iH_1)))(d-m)\cong
	\mathcal{O}_{J(X_f,q)}(d-m-i)\). This gives us the first isomorphism for \(m>2\). If
	\(m=2\), the first arrow is defined similarly but it is not surjective onto
	\(\mathcal{H}^0(\Xi_{d-m,0}(\mathcal{O}_{X_f}(-iH_1)))\). 

  For \(\mathcal{H}^{m-2}\) we need to compute
  \(\mathbf{R}^{m-1}\beta\mathcal{O}_{[\Bl/\mu_d]}( -(d+i)H + (d+i)E)\). By
	Grothendieck duality and the derived functor spectral sequence, we have an isomorphism
	\begin{align*}
    &\mathbf{R}^{m-1}\beta_\ast\mathcal{O}_{[\Bl/\mu_d]}( -(d+i)H + (d+i)E) \\
		&\cong \mathbf{R}\mathcal{H}om_{[\mathbb{P}^n/\mu_d]}(
		\mathbf{R}\beta_\ast(\mathcal{O}_{[\Bl/\mu_d]}(m-1-d-i)E),
		\mathcal{O}_{[\mathbb{P}^n/\mu_d]})(-d-i) \\
		&\cong \mathbf{R}\mathcal{H}om_{[\mathbb{P}^n/\mu_d]}(\mathcal{I}_q^{d-m+i+1},
		\mathcal{O}_{[\mathbb{P}^n/\mu_d]})(-d-i).
	\end{align*}
	and the second isomorphism follows by twisting by \( (d-m)\).
\end{proof}

\begin{corollary}
	For \(i = 0,\ldots,d-m\), we have
	\(\mathcal{H}^{m-1}((\mathcal{I}_q^{d-m+i+1})^\vee)(-m-i)\in\mathcal{T}\).
	\label{cor:reduction-to-torsion}
\end{corollary}

\begin{proof}
  Consider the exact triangle of Lemma \ref{lem:cohomology-sheaves-functor}. By
  Lemma \ref{lem:join-xf-q}, the first term of this triangle is in
  \(\mathcal{T}\). By Proposition \ref{prop:images}, the middle term is in
  \(\mathcal{T}\). Hence, the third term is in \(\mathcal{T}\) too.
\end{proof}

It remains to compute
\(\mathcal{H}^{m-1}((\mathcal{I}_q^{d-m+i+1})^\vee)(-m-i)\) for \(i =
0,\ldots,m-1\).

\subsubsection{Derived duals of powers of \(\mathcal{I}_q\).}

We seek to understand the sheaves in Corollary \ref{cor:reduction-to-torsion}.
To that end, there is an exact sequence of sheaves on \(\mathbb{P}^m\), where we
have identified the conormal bundle of \(\{q\}\) with
\(\Omega_{\mathbb{P}^m,q}\):
\[
	0\to \mathcal{I}_q^{r+1}\to \mathcal{I}_q^r\to\mathcal{O}_q^{\oplus N(r)}\otimes\chi^r\to 0.
\]
where \(N(r) = {m+r-1\choose r}\) and we have used the identification
\[
  S^r(\Omega_{\mathbb{P}^m,q})\cong \mathcal{O}_q^{\oplus N(r)}\otimes\chi^r.
\]
Since \(\mathcal{O}_q\) is a smooth closed
subscheme of codimension \(m\), we know that \(
(\mathcal{O}_q\otimes\chi)^\vee\) is concentrated in degree \(m\) and
\[
	\mathcal{H}^m( (\mathcal{O}_q\otimes\chi^r)^\vee)\cong \mathcal{O}_q\otimes
	\omega_{\mathbb{P}^m}^\vee\otimes\chi^{-r}
\]
since \(\omega_{[\mathbb{P}^m/\mu_d]}\cong
\mathcal{O}_{\mathbb{P}^m}(-m-1)\otimes\chi^{-1}\), we see
\[
	\mathcal{H}^m( (\mathcal{O}_q\otimes\chi^r)^\vee)\cong
	\mathcal{O}_q\otimes\chi^{-r-m}.
\]

Taking the derived dual of the above sequence now yields the short exact
sequence
\[
	0\to \mathcal{H}^{m-1}((\mathcal{I}_q^r)^\vee)\to \mathcal{H}^{m-1}(
	(\mathcal{I}_q^{r+1})^\vee)\to \mathcal{O}_q^{\oplus N(r)}\otimes
	\chi^{-r-m}\to 0.
\]

\begin{lemma}
  For \(r>0\), the sheaves \(\mathcal{H}^{m-1}( (\mathcal{I}_q^{r+1})^\vee)\) have a
	filtration by sheaves
	\(0=\mathcal{F}_0\subset\mathcal{F}_1\subset\cdots\subset\mathcal{F}_{r+1} =
	\mathcal{H}^{m-1}( (\mathcal{I}_q^r)^\vee)\) such that
	\[
		\mathcal{F}_i/\mathcal{F}_{i+1}\cong
		\mathcal{O}_q^{\oplus N(r)}\otimes\chi^{-i-m}.
	\]
	In particular, \(\mathcal{H}^{m-1}( (\mathcal{I}_q^{r+1})^\vee)\in\langle
	\mathcal{O}_q\otimes\chi^{-m},\mathcal{O}_q\otimes\chi^{-1-m},\ldots,\mathcal{O}_q\otimes\chi^{-r-m}\rangle\).
	\label{lem:derived-dual-powers-ideal}
\end{lemma}

\begin{proof}
	Immediate from the previous discussion and the observation
	\(\mathcal{H}^{m-1}(\mathcal{I}_q^\vee)\cong
	\mathcal{O}_q\otimes\chi^{-m}\).
\end{proof}

Using Lemma \ref{lem:derived-dual-powers-ideal}, we have the following
filtration of \(\mathcal{H}^{m-1}( (\mathcal{I}_q^{d-m+1+i})^\vee)\):
\[
\begin{tikzcd}
	\mathcal{H}^{m-1}( (\mathcal{I}_q^{d-m+1+i})^\vee)(d-m-i)\ar{rr} &&
	\mathcal{O}_q^{\oplus N(d-m+i)}\otimes\chi^{-(d-m)} \\
	\mathcal{H}^{m-1}( (\mathcal{I}_q^{d-m+i})^\vee)(d-m-i)\ar{rr} \ar{u}&&
	\mathcal{O}_q^{\oplus N(d-m-1+i)}\otimes\chi^{-(d-m)+1} \\
	\vdots\ar{u} && \vdots \\
	\mathcal{H}^{m-1}( (\mathcal{I}_q^2)^\vee)(d-m-i)\ar{rr} \ar{u}&&
	\mathcal{O}_q^{\oplus N(1)}\otimes\chi^{-1+i} \\
	\mathcal{H}^{m-1}( (\mathcal{I}_q)^\vee)(d-m-i)\ar{rr} \ar{u}&&
	\mathcal{O}_q\otimes\chi^{i} \\
\end{tikzcd}
\]

\begin{lemma}
	For all \(i=0,\ldots,d-1\) we have
	\(\mathcal{O}_q\otimes\chi^i\in\mathcal{T}\).
	\label{lem:missing-torsion-q}
\end{lemma}

\begin{proof}
	We proceed by induction on \(i\). When \(i = 0\) we use the filtration given
	by Lemma \ref{lem:derived-dual-powers-ideal} and the fact that
	\(\mathcal{O}_q\otimes\chi^{-j}\in\mathcal{T}\) for \(j = 1,\ldots,d-m\) to
	see \(\mathcal{O}_q\in\mathcal{T}\). Suppose we have
	\(\mathcal{O}_q,\ldots,\mathcal{O}_q\otimes\chi^i\), then we use the
	filtration again with \(i+1\) to see
	\(\mathcal{O}_q\otimes\chi^{i+1}\in\mathcal{T}\).
\end{proof}

\begin{lemma}
	For all \(i = 0,\ldots,d-1\) we have
	\(\mathcal{O}_p\otimes\chi^i\in\mathcal{T}\).
	\label{lem:missing-torsion-p}
\end{lemma}

\begin{proof}
  This is analagous to the results proved in this section using the images
  \(\Xi_{d-n,d-n}(\mathcal{O}_{X_g}(-jH_2))\). The only difference is that the
  hyperplane section that restricts to \(H_2\) has nontrivial equivariant
  structure, see Remark \ref{rem:equivariant-structure}.
\end{proof}

\begin{proof}[Proof of Main Result]
  We have shown in \S \ref{ssec:free-orbits} the structure sheaves of free orbits away from the join of
  \(X_f\) and \(X_g\) are in
  \(\mathcal{T}\). By Corollary \ref{cor:free-orbit-join}, the structure sheaves
  of free orbits along the join are also in \(\mathcal{T}\). By Lemmas \ref{lem:missing-torsion-q} and
	\ref{lem:missing-torsion-p} we have that structure sheaves of all twists
  of fixed points are in \(\mathcal{T}\). Thus \(\mathcal{T}\) has a spanning
  class of \(\mathcal{D}[X/\mu_d]\). Since \(\mathcal{T}\) is saturated hence admissible, by Proposition
	\ref{prop:subcat-spanning-class} we conclude \(\mathcal{T} =
	\mathcal{D}[X/\mu_d]\).
\end{proof}

\section{Comparison with Orlov's functors}
\label{sec:orlov}

In this section, we show that in the case of two Calabi-Yau hypersurfaces, i.e.
\(m=n=d\), our functor \(\Xi_{0,0}\) agrees with Orlov's up to a mutation and a twist
by a line bundle. We first recall the relationship between matrix factorization
and singularity categories. Then we discuss Orlov's theorem in detail and use
this discussion to show that the functors agree.

\subsection{Graded Matrix Factorizations and Graded Singularity Categories}

For a detailed account of the relationship between graded matrix factorization
and graded singularity categories see \cite[Thm. 39]{orlov-sing-09} and \cite{bfk1-14}.

Let \(V\) be a finite dimensional vector space over \(k\) of dimension \(n\).
Set \(R = \mathrm{Sym}(V^\vee)\).  Then \(R\) is \(\mathbb{Z}\)-graded, where
\(V^\vee\) sits in degree 1. Let \(f\in R_d\) define a smooth projective
hypersurface \(X:=V(f) \subset\mathbb{P}(V)\). Set \(A = R/(f)\) to be the
hypersurface algebra.

\begin{definition}
  A \textbf{graded matrix factorization} of \(f\) is a pair of morphisms
  \[
    \delta_0\colon P_{-1}\to P_0,\ \delta_{-1}\colon P_0\to P_{-1}
  \]
  between graded projective \(R\)-modules such that
  \[
    \delta_0\delta_{-1} = f = \delta_{-1}\delta_0.
  \]
  \label{def:graded-matrix-factorization}
\end{definition}

Morphisms of graded matrix factorizations are morphisms of the underlying graded
modules that make the relevant diagrams commute. There is a notion of homotopy
between two morphisms and we set \(\mathrm{HMF}^{gr}(f)\) to be the
corresponding homotopy category. This category is also triangulated.

Now let \(\mathrm{gr}-A\) denote the category of graded \(A\)-modules. Set
\(\mathrm{grproj}-A\) to be the full subcategory generated by graded projective
\(A\)-modules. This is a thick subcategory. 

\begin{definition}
  The Drinfield-Verdier quotient of \(\mathcal{D}(\mathrm{gr}-A)\) by
  \(\mathcal{D}(\mathrm{grproj}-A)\) is the \textbf{graded singularity category}
  of \(A\), denoted \(\mathcal{D}^{gr}_{Sg}(A)\). In other words,
  \[
    \mathcal{D}(\mathrm{grproj}-A)\to \mathcal{D}(\mathrm{gr}-A)\to
    \mathcal{D}^{gr}_{Sg}(A)
  \]
  is an exact sequence of triangulated categories.
\end{definition}

There is a functor between the two categories we have introduced called the
cokernel functor:
\[
  \mathrm{cok}\colon \mathrm{HMF}^{gr}(f)\to \mathcal{D}^{gr}_{Sg}(A)
\]
which sends a graded matrix factorization \( (\delta_{-1},\delta_0)\) to the graded
\(A\)-module \(\mathrm{cok}(\delta_0)\). 

\begin{theorem}[\protect{\cite{orlov-sing-09}}]
  The functor \(\mathrm{cok}\) is well-defined and an equivalence of
  triangulated categories.
  \label{thm:cok-functor}
\end{theorem}

\begin{definition}
  Let \(\mathrm{stab}\colon \mathcal{D}^{gr}_{Sg}(A)\to \mathrm{HMF}^{gr}(f)\) denote
  the quasi-inverse to \(\mathrm{cok}\), call it the \textbf{stabilization
  functor}. For the image of a graded module \(M\) in \(D^{gr}_{Sg}(A)\), denote
  by \(M^{stab}\) its image under \(\mathrm{stab}\), i.e.
  \(\mathrm{stab}(M) =
  M^{stab}\).
\end{definition}

\begin{remark}
  If the grading is given by some extension of \(\mathbb{G}_m\) by a finite
  group, everything still goes through as in the \(\mathbb{G}_m\) case,
  \cite[Remark 3.65]{bfk1-14}.
\end{remark}

\subsection{Orlov's Theorem}

Let \(\mathrm{tr}_{\geq
i}\) to be the truncation endofunctor on \(\mathrm{gr}-A\). Set
\(\mathrm{gr}-A_{\geq i}\) to be the image of \(\mathrm{gr}-A\) under
\(\mathrm{tr}_{\geq i}\).

Define \(S_{<i}\) to be the full triangulated subcategory of
\(\mathcal{D}(\mathrm{gr}-A)\) generated by the finite dimensional graded
\(A\)-modules \(k(e)\) for \(e>-i\). Equivalently, this is the kernel of
\(\mathrm{tr}_{\geq i}\). Similarly define \(S_{\geq i}\).

Let \(P_{<i}\) be the full triangulated subcategory of
\(\mathcal{D}(\mathrm{gr}-A)\) generated by the projective \(A\)-modules \(
A(e)\) for \(e>-i\). Similarly define \(P_{\geq i}\). Let \(\mathrm{tors}-A\) be
the full subcategory of \(\mathrm{gr}-A\) generated by finite dimensional
\(A\)-modules.

We also have the projection functor
\[
  \pi_i\colon \mathcal{D}(\mathrm{gr}-A_{\geq i})\to
  \mathcal{D}^{gr}_{Sg}(A)
\]
which has kernel \(P_{\geq i}\) and induces an exact equivalence
\[
  \mathcal{D}(\mathrm{gr}-A_{\geq i})/P_{\geq i}\simeq
  \mathcal{D}_{Sg}^{gr}(A).
\]
Moreover, there is a semi-orthogonal decomposition
\[
  \mathcal{D}(\mathrm{gr}-A_{\geq i}) = \langle P_{\geq i}, T_i\rangle
\]
where \(T_i\) is equivalent to the graded singularity category via \(\pi_i\).
Let \(\pi_i^{-1}\) be the quasi-inverse to \(\pi_i\) restricted to \(T_i\).

Orlov defines two \(\mathbb{Z}\)-indexed families of functors
\[
  \Psi_i\colon \mathcal{D}(X)\to \mathrm{HMF}^{gr}(f)
\]
and
\[
  \Phi_i\colon \mathrm{HMF}^{gr}(f)\to \mathcal{D}(X)
\]
as follows.

The functors \(\Psi_i\) are the composite
\[
  \Psi_i:=\mathrm{stab}\circ\pi_i\circ\mathrm{tr}_{\geq i-n+d}\circ
  \Gamma_\ast.
\]
Here \(\Gamma_\ast\) is right derived graded global sections. The functors
\(\Phi_i\) are the composite
\[
  \Phi_i:=\mathrm{sh}\circ\pi_i^{-1}\circ\mathrm{stab}.
\]
Here \(\mathrm{sh}\) is the sheafification functor.

\begin{theorem}[\protect{\cite{orlov-sing-09}}]
  Set \(a = n-d\). For each \(i\in\mathbb{Z}\), the triangulated categories \(\mathcal{D}(X)\) and
  \(\mathcal{D}^{gr}_{Sg}(A)\) are related as follows. 
  \begin{enumerate}[(i)]
    \item If \(a>0\), \(\Phi_i\) is fully-faithful and there is a
      semi-orthogonal decomposition
      \[
        \mathcal{D}(X) = \langle
        \mathcal{O}_X(-i-a-1),\ldots,\mathcal{O}_X(-i),
        \Phi_i\mathcal{D}^{gr}_{sg}(A)\rangle;
      \]
    \item If \(a<0\), \(\Psi_i\) is fully-faithful and there is a
      semi-orthogonal decomposition 
      \[
        \mathcal{D}^{gr}_{Sg}(A) = \langle k^{stab}(-i),\ldots,
        k^{stab}(-i+a+1),\Psi_i\mathcal{D}(X)\rangle;
      \]
    \item If \(a = 0\), \(\Psi_i\) and \(\Phi_i\) are mutually-inverse
      equivalence of categories.
  \end{enumerate}
  \label{thm:orlovs-thm-sing-dx}
\end{theorem}

\begin{remark}
  The theorem still holds when \(A\) is graded by \(\mathbb{Z}\times\mu_d\),
  \cite[Thm. 3.25]{bfk2-14}.
\end{remark}

\subsection{Comparison with Orlov's Functors}

Let \(X_f\) and \(X_g\) be smooth Calabi-Yau hypersurfaces, so \(m=n=d\). Set
\(A_f\) and \(A_g\) to be the corresponding hypersurface algebras, respectively,
and \(A\) to be the hypersurface algebra of \(X = V(f\oplus
g)\subset \mathbb{P}^{2n-1}\). We consider the following functor defined using
Orlov's functors:
\[
  \Omega=\Phi_0\circ \Psi_{0,0}\colon \mathcal{D}(X_f\times X_g)\to \mathcal{D}[X/\mu_d]
\]
where \(\Psi_{0,0}\) is the embedding:
\[
  \Psi_{0,0}:= \Psi_0\otimes\Psi_0\colon \mathcal{D}(X_f)\otimes
  \mathcal{D}(X_g)\to \mathrm{HMF}^{gr}(f)\otimes \mathrm{HMF}^{gr}(g)
  \cong\mathrm{HMF}^{gr,\mu_d}(f\oplus g).
\]

Let \(p = [p_1:\cdots:p_n]\in X\) be such that \(p_i\neq 0\). Let \(l_p\)
denote the graded \(A\)-module given by \(\mathrm{tr}_{\geq 0}\circ \Gamma_\ast
(\mathcal{O}_p)\). Then there is an isomorphism
\[
  l_p\cong A/(x_1-p_1x_i,\ldots,x_n-p_nx_i).
\]
The stabilization of \(l_p\) is given by a Koszul matrix factorization,
see \cite[\S 1.6]{polish-vaintrob-cft} for the definition, say
\(l_p^{stab}\). Hence \(\Psi_0(\mathcal{O}_p) \cong l_p^{stab}\).

\begin{lemma}
  For \(p\in X_f\) and \(q\in X_g\), we have
  \[
    \Psi_{0,0}(\mathcal{O}_{(p,q)})\cong S_{p,q}^{stab}
  \]
  where \(S_{p,q}\) is the \(\mathbb{Z}\times
  \mu_d\)-graded \(A\)-module corresponding to the structure sheaf
  of the unique plane containing \(p\) and \(q\).
  \label{lem:planes-join-lines}
\end{lemma}

\begin{proof}
  Immediate from the previous discussion and the isomorphisms
  \[
    \Psi_{0,0}(\mathcal{O}_{(p,q)})\cong
    \Psi_0(\mathcal{O}_p)\otimes\Psi_0(\mathcal{O}_q)\cong l_p^{stab}\boxtimes
    l_q^{stab}\cong S_{p,q}^{stab}.
  \]
  where the last isomorphism comes from the fact that the Koszul matrix
  factorizations are given by a regular sequence.
\end{proof}

We now need to compute \(\Phi_0(S_{p,q}^{stab})\). To do this we first make a
digression to note that the decomposition in the Main Theorem:
\[
  \mathcal{D}[X/\mu_n] = \langle \mathcal{O}_X(-2(n-1))(\chi^{-(n-1)}),\ldots,
  \mathcal{O}_X(-1)(\chi^{0,-1}),\mathcal{O}_X,\Xi_{0,0}\mathcal{D}(X_f\times
  X_g)\rangle
\]
doesn't quite match Theorem \ref{thm:orlovs-thm-sing-dx}:
\[
  \mathcal{D}[X/\mu_n] = \langle
  \mathcal{O}_X(-(n-1))(\chi^{0,\ldots,n-1}),\ldots,\mathcal{O}_X(\chi^{0,\ldots,n-1}),
  \Omega\mathcal{D}(X_f\times X_g)\rangle.
\]

To that end, we set \(\mathcal{A}_{\geq -(n-1)}\) to be the full subcategory
generated of \(\mathcal{D}[X/\mu_d]\) generated by
\[
  \mathcal{O}_X(-(n-1))(\chi^{0,\ldots,-(n-1)}),\ldots,\mathcal{O}_X(\chi^{0,-1}),\mathcal{O}_X
\]
and \(\mathcal{A}_{<-(n-1)}\) to be the one generated by
\[
  \mathcal{O}_X(-2(n-1))(\chi^{-(n-1)}),\ldots,
  \mathcal{O}_X(-n)(\chi^{-1,\ldots,-(n-1)})
\]
so that the decomposition from Theorem \ref{thm:main-result} is
\[
  \mathcal{D}[X/\mu_n] = \langle \Xi_{0,0}\mathcal{D}(X_f\times
  X_g),\mathcal{A}_{<-(n-1)},\mathcal{A}_{\geq -(n-1)}\rangle.
\]

To match the decomposition of Theorem \ref{thm:orlovs-thm-sing-dx} we perform
some mutations and use the Serre functor to see that
\[
  \mathcal{D}[X/\mu_n] = \langle \mathcal{O}_X(-(n-1))(\chi^{all}),\ldots,
  \mathcal{O}_X(\chi^{all}),
  R_{\mathcal{A}_{<-(n-1)}(n)}(\Xi_{0,0}(\mathcal{D}(X_f\times X_g)))\rangle.
\]

Since \(\mathrm{cok}(S_{p,q}^{stab}) \cong S_{p,q}\) and the sheafification is
\(\mathcal{O}_{l(p,q)} = \Xi_{0,0}(\mathcal{O}_{(p,q)})\), we will be done if we
can see that the mutations that we have to do in \(\mathcal{D}[X/\mu_d]\) agree
with those that we need to do to mutate \(S_{p,q}\) into \(T_0\). Indeed, the
object \(S_{p,q}\) is already left orthogonal to \(S_{<0}\), so we just need to
find a representative in \(T_0\) mapping onto \(S_{p,q}\).

\begin{lemma}
  We have the vanishing
  \[
    \mathrm{Ext}^\ast_{gr-A}(S_{p,q},A(e)(\chi^i))^{\mu_d}=0
  \]
  for \(-n+1\geq e \geq 0\) and \(e\leq i\leq 0\).
  \label{lem:vanishing-spq}
\end{lemma}

\begin{proof}
  Using the isomorphism
  \[
    \mathrm{Ext}^\ast_{gr-A}(M,A)\cong
    \mathrm{Ext}^{\ast+1}_{gr-R}(M,R(-n))
  \]
  we have the following isomorphism computed by taking a Koszul complex
  \[
    \mathrm{Ext}^\ast_A(S_{p,q},A(e)(\chi^i)) \cong
    S_{p,q}(n-2+e)(\chi^{i-1})[3-2n].
  \]
  The statement follows from taken degree 0 peices and \(\mu_d\)-invariants.
\end{proof}

\begin{lemma}
  There is an isomorphism
  \[
    \mathrm{Ext}^\ast_{gr-A}(S_{p,q},A(e)(\chi^i)) \cong
    \mathrm{Ext}_X^\ast(\mathcal{O}_{l(p,q)},\mathcal{O}_X(e)(\chi^i))
  \]
  for \(-n+1\leq e\leq 0\) and \(-n\leq i<e\).
  \label{lem:equality-ext-spq}
\end{lemma}

\begin{proof}
  We compute:
  \begin{align*}
    \mathrm{Ext}^\ast_X(\mathcal{O}_{l(p,q)},\mathcal{O}_X(e)(\chi^i))&\cong
    \mathrm{Ext}^{2n-2-\ast}_X(\mathcal{O}_X(e)(\chi^i),\mathcal{O}_X(-n)) \\
    &\cong H^{2n-2-\ast}(\mathcal{O}_{l(p,q)}(-e-n)(\chi^{-i})) \\
    &\cong H^{3-2n+\ast}(\mathcal{O}_{l(p,q)}(n+e-2)(\chi^{i-1})) \\
    &\cong (S_{p,q}(n+e-2)(\chi^{i-1}))_0[3-2n]
  \end{align*}
  
  The claim follows from the computation in Lemma \ref{lem:vanishing-spq}
\end{proof}

\begin{theorem}
  For a point \( (p,q)\in X_f\times X_g\), we have
  \[
    \Omega(\mathcal{O}_{(p,q)})\cong
    (\mathbf{R}_{\mathcal{A}_{<-(n-1)}(n)}\circ\Xi_{0,0})(\mathcal{O}_{p,q}).
  \]
  \label{thm:agrees-points}
\end{theorem}

\begin{proof}
  Follows from Lemmas \ref{lem:vanishing-spq} and \ref{lem:equality-ext-spq}.
\end{proof}

\begin{corollary}
  The functors \(\Xi_{0,0}\) and \(\Omega\) agree up to a twist by a line
  bundle.
\end{corollary}

\begin{proof}
  Since the image of structure sheaves of closed points agree by Theorem
  \ref{thm:agrees-points}, the images of \(\Xi_{0,0}\) and \(\Omega\) agree.
  Denote also by \(\Omega\), the restriction of \(\Omega\) to its image so that
  it is an equivalence. Then the composite functor \(\Omega^{-1}\circ\Xi_{0,0}\)
  is an endofunctor of \(\mathcal{D}(X_f\times X_g)\) satisfying
  \[
    \Omega^{-1}\circ\Xi_{0,0}(\mathcal{O}_{p,q})\cong\mathcal{O}_{p,q}.
  \]
  By\cite[Corollary 5.23]{huybrechts-fourier}, there exists a line bundle
  \(\mathcal{L}\) on \(X_f\times X_g\) such that
  \[
    \Omega^{-1}\circ \Xi_{0,0}\cong (?)\otimes\mathcal{L}
  \]
  and hence
  \[
    \Xi_{0,0}\cong \Omega( (?)\otimes\mathcal{L}).
  \]
\end{proof}

\section{Special Cases}
\label{sec:special-cases}

For completeness, we devote this section to understanding \([X/\mu_d]\) when \(m
= 1\). We will independently study \(n=1\) and \(n>1\). In the case \(n>1\), the
hypersurface \(X\) is called a \textit{cyclic hypersurface}. We also compare the
decompositions to the work in \cite{kuznetsov-perry}

\subsection{The case \(n=1\).}
\label{ssec:case-n1}

Let \( m,n=1\), then \(\mathcal{D}(X_f)\) and \(\mathcal{D}(X_g)\) will not
appear. The associated quotient stack is, up to a change of coordinates, \(X =
V(x^d+y^d)\subset\mathbb{P}^1\). In particular, \(|X| = d\) and the \(\mu_d\)
action permutes the linear factors. It follows that \([X/\mu_d]\) is a scheme
and is represented by \(\Spec(k)\). There is a single exceptional object given
by \(\mathcal{O}_X\) and so \[ \mathcal{D}[X/\mu_d] = \langle
\mathcal{O}_X\rangle.  \]

\subsection{The case \(n>1\).}

Let \(n>1\). Then \(f(x) = x^d\) and \(g(y_1,\ldots,y_n)\)
be a degree \(d\) polynomial defining a smooth hypersurface in
\(\mathbb{P}^{n-1}\). Let \(\pi\colon X\to \mathbb{P}^{n-1}\) be the linear projection
onto the \(y\) variables. This is well defined since \([1:0:\ldots:0]\notin X\).
The map \(\pi\) is a degree \(d\) mapping ramified along the divisor
\(\iota_g\colon X_g\hookrightarrow X\). In particular, we have the following
commutative diagram.
\[
	\begin{tikzcd}
		{} & X\ar{d}{\pi} \\
		X_g\ar{ur}{\iota_g} \ar{r} & \mathbb{P}^{n-1}
	\end{tikzcd}
\]
Endowing \(\mathbb{P}^{n-1}\) with the trivial \(\mu_d\) action renders the
diagram \(\mu_d\)-equivariant. Moreover, it is not hard to see that \(\pi\)
exhibits \(\mathbb{P}^{n-1}\) as a coarse moduli space. 

\begin{theorem}
	There is a semi-orthogonal decomposition
	\[
		\mathcal{D}[X/\mu_d] = \langle \mathcal{D}_g^1,\ldots,
		\mathcal{D}_g^{d-1}, \pi^\ast\mathcal{D}(\mathbb{P}^{n-1})\rangle,
	\]
	where \(\mathcal{D}_g^i = \iota_{g\ast}\mathcal{D}(X_g)\otimes\chi^i\).
	\label{thm:sod-m1}
\end{theorem}

\begin{proof}
	The action of \(\mu_d\) on \(X\setminus X_g\) is easily seen to be free and
  the action of \(\mu_d\) on the normal bundle to the embedding
  \(X_g\hookrightarrow X\) has constant weight \(\chi\). The result follows from
  a quick adaptation of \cite[Thm 4.1]{kuznetsov-perry}.
\end{proof}

\subsection{Derived Categories of Cyclic Covers.}

The case of a cyclic cover of a variety was investigated in \cite[\S
8.3]{kuznetsov-perry}. In particular, they discuss the equivariant derived
category of cyclic hypersurfaces where \(d\leq n\), here \(X\subset
\mathbb{P}^n\) and \(X_g\subset\mathbb{P}^{n-1}\). For completeness, we recall
their result. Since \(d\leq n\), we have the standard semi-orthogonal
decomposition of a hypersurface
\[
	\mathcal{D}(X) = \langle \mathcal{A}_X,\mathcal{O}_X,\ldots,
	\mathcal{O}_X(d-n)\rangle,
\]
where \(\mathcal{A}_X\) is characterized as the right orthogonal to \(\langle
\mathcal{O}_X,\ldots,\mathcal{O}_X(d-n)\rangle\). The category
\(\mathcal{A}_X\) is also quasi-equivalent to the homotopy category of graded matrix
factorizations of the potential \(f\), where \(f\) is the defining equation for
\(X\).

\begin{theorem}
	In the above notation, if \(d\leq n\), then there is a decomposition
	\[
		\mathcal{D}[X/\mu_d] = \langle
		\mathcal{A}_X^{\mu_d},\mathcal{O}_X\otimes\chi^{0,\ldots,d-1},\ldots
		\mathcal{O}_X(d-n)\otimes \chi^{0,\ldots,d-1}\rangle
	\]
	where
	\[
		\mathcal{A}_X^{\mu_d} = \langle \mathcal{A}_{X_g},
		\mathcal{A}_{X_g}\otimes\chi,\ldots,
		\mathcal{A}_{X_g}\otimes\chi^{n-2}\rangle.
	\]
	where \(\mathcal{A}_{X_g}\) is viewed as a subcategory of
	\(\mathcal{D}[X/\mu_d]\) via \(\iota_{g\ast}\).
	\label{thm:kuznetsov-perry-cyclic-hypersurfaces}
\end{theorem}

\begin{remark}
 	This result does not apply when \(d>n\) because \(\pi\colon
	X\to\mathbb{P}^{n-1}\) is not a cyclic cover in the sense of
	\cite{kuznetsov-perry}. 
\end{remark}

When \(d = n\) the subcategory \(\mathcal{A}_{X_g}\) is all of
\(\mathcal{D}(X_g)\). Using the notation \(\mathcal{D}_g^i=
\iota_{g\ast}(\mathcal{D}(X_g))\otimes\chi^i\), the decomposition of Theorem
\ref{thm:kuznetsov-perry-cyclic-hypersurfaces} is
\[
	\mathcal{D}[X/\mu_d] = \langle \mathcal{D}_g^0, \mathcal{D}_g^1,
	\ldots, \mathcal{D}_g^{n-2}, \mathcal{O}_X\otimes \chi^{0,\ldots,n-1}\rangle.
\]
The decomposition of Theorem \ref{thm:sod-m1} is
\[
	\mathcal{D}[X/\mu_d] = \langle \mathcal{D}_g^1,\ldots,
	\mathcal{D}_g^{n-1}, \pi^\ast(\mathcal{D}(\mathbb{P}^{n-1}))\rangle.
\]
It follows that our decomposition agrees with theirs up to a twist by a
character.

\end{document}